\documentclass[11pt, oneside]{article}

\usepackage{geometry}                		
\usepackage{graphicx}				

\usepackage[usenames]{xcolor}		
\usepackage{amsmath,amsthm,amssymb}
\usepackage{amscd}

\usepackage[all]{xy}
\usepackage{tikz}

\theoremstyle{definition}
\newtheorem{dfn}{Definition}[section]

\theoremstyle{plain}
\newtheorem{thm}[dfn]{Theorem}
\newtheorem{prop}[dfn]{Proposition}
\newtheorem{lem}[dfn]{Lemma}
\newtheorem{cor}[dfn]{Corollary}

\newcommand\ABG{\omega_{\mathrm{ABG}}}
\newcommand\ABGtr{\omega^{\text{\rm trace}}_{\mathrm{ABG}}}
\newcommand\Ad{\mathrm{Ad}}

\newcommand\bC{\mathbb{C}}
\newcommand\bD{{\overline{D}}}
\newcommand\bH{\mathbb{H}}
\newcommand\bL{\mathbb{L}}
\newcommand\bp{\mathbf{p}}
\newcommand\bR{\mathbb{R}}
\newcommand\Btr{B^{\text{trace}}}
\newcommand\bZ{\mathbb{Z}}
\newcommand\cB{\mathcal{B}}

\newcommand\cP{\mathcal{P}}
\newcommand\cS{\mathcal{S}}

\newcommand\inv{^{-1}}
\newcommand\la{{\lambda}}
\newcommand\op{^{\mathrm{op}}}
\newcommand\pa{\partial}
\newcommand\PSL{PSL_2(\mathbb{R})}

\newcommand{\rot}{\mathrm{rot}}
\newcommand\Sg{{\Sigma_g}}
\newcommand\SL{SL_2(\mathbb{R})}
\newcommand\SP{{\Sigma_{\mathcal{P}}}}
\newcommand\sltwo{{\mathfrak{sl}_2(\mathbb{R})}}
\newcommand\sltwoc{{\mathfrak{sl}_2(\mathbb{C})}}
\newcommand\Tg{\mathcal{T}_g}
\newcommand\trace{\mathrm{trace}}
\newcommand\WP{\omega_{\mathrm{WP}}}
\newcommand\zb{\overline{z}}

\newcommand{\twovector}[2]{\begin{pmatrix}{#1}\\ {#2}\end{pmatrix}}
\newcommand{\twomatrix}[4]{\begin{pmatrix}
{#1} & {#2} \\ {#3} & {#4}\end{pmatrix}}

\title{A topological proof of Wolpert's formula 
for the Weil-Petersson symplectic form 
in terms of the Fenchel-Nielsen coordinates
}
\author{Nariya Kawazumi\thanks{Department of Mathematical Sciences, University of Tokyo, 3-8-1 Komaba, Meguro-ku, Tokyo 153-8914, Japan \texttt{e-mail:kawazumi@ms.u-tokyo.ac.jp}}
}

\begin{document}

\maketitle

\begin{abstract}
We introduce a natural cell decomposition of a closed oriented surface associated with a pants decomposition, and 
an explicit groupoid cocycle on the cell decomposition which represents each point of the Teichm\"uller space $\Tg$. We call it the {\it standard cocycle}
of the point of $\Tg$. As an application of the explicit description of 
the standard cocycle, we obtain a topological proof of Wolpert's formula for the Weil-Petersson symplectic form in terms of the Fenchel-Nielsen coordinates associated with the pants decomposition.
\end{abstract}

\begin{center}
{\bf Introduction}
\end{center}

The Fenchel-Nielsen coordinates $(l_i, \tau_i)^{3g-3}_{i=1}: \Tg \to 
(\bR_{>0}\times\bR)^{3g-3}$ on the Teichm\"uller space $\Tg$ of an oriented closed surface $\Sg$ of genus $g \geq 2$ is associated with a pants decomposition $\cP = \{\delta_i
\}^{3g-3}_{i=1}$ of the surface $\Sg$, and 
is a diffeomorphism onto the space $(\bR_{>0}\times\bR)^{3g-3}$,
{where $\delta_i$'s are mutually disjoint simple closed curves on $\Sg$.}
Wolpert \cite{W1, W2, W3} described the Weil-Petersson symplectic form 
$\WP$ on the space $\Tg$ in terms of the coordinates
\begin{equation}\label{eq:Wformula}
\WP = \sum^{3g-3}_{i=1} d\tau_i\wedge dl_i.
\end{equation}
\par
In the present paper, we introduce 
a natural cell decomposition $\SP$ of the surface $\Sg$ associated with the decomposition $\cP$ {as in Figure 1,} and an explicit groupoid $1$-cocycle on the cell decomposition $\SP$ which represents each point of the Teichm\"uller space $\Tg$. 

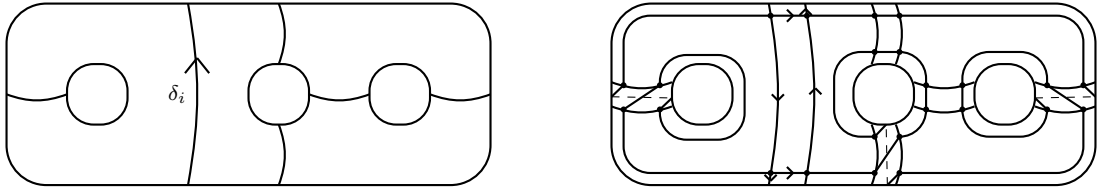
\begin{figure}[h]
\begin{center}
\begin{tikzpicture}[scale=0.8]
\begin{scope}
\tikzset{Process/.style={rectangle,  draw,  text centered, text width=37mm, minimum height=1.5cm}};
\draw[thick, rounded corners=15pt] (1,0) -- (8,0) -- (8,3) -- (0,3) -- (0,0) -- (1,0);
\draw[thick, rounded corners=10pt] (1,1) rectangle (2,2);
\draw[thick, rounded corners=10pt] (4,1) rectangle (5,2);
\draw[thick, rounded corners=10pt] (6,1) rectangle (7,2);
\draw[thick] (0,1.5) to[out=-20, in=-160] (1,1.5);
\draw[thick] (5,1.5) to[out=-20, in=-160] (6,1.5);
\draw[thick] (7,1.5) to[out=-20, in=-160] (8,1.5);
\draw[thick] (4.5,3) to[out=-70, in=70] (4.5,2);
\draw[thick] (4.5,1) to[out=-70, in=70] (4.5,0);
\draw[thick] (3,3) to[out=-80, in=80]  node[auto=right]{\scriptsize $\delta_i$} (3,0);
\draw[thick] (2.95,1.85) -- (3.15,2.1);
\draw[thick] (3.35,1.85) -- (3.15,2.1);
\end{scope}
\begin{scope}[xshift=10cm]
\draw[thick, rounded corners=15pt] (1,0) -- (8,0) -- (8,3) -- (0,3) -- (0,0) -- (1,0);
\draw[thick, rounded corners=10pt] (1,1) rectangle (2,2);
\draw[thick, rounded corners=10pt] (4,1) rectangle (5,2);
\draw[thick, rounded corners=10pt] (6,1) rectangle (7,2);
\draw[thick] (0,1.7) to[out=-20, in=-160] (1,1.7);
\draw[thick] (0,1.3) to[out=-20, in=-160] (1,1.3);
\draw[thick] (0,1.46) -- (0.2,1.65); 
\draw[thick] (0.8,1.25) -- (1,1.44); 
\draw[dashed] (0,1.46) -- (1.0,1.44); 
\draw[thick] (0.2,1.25) -- (0.8,1.65); 
\draw[thick, rounded corners=10pt] (0.8,1.65) -- (0.8,2.15) --
(2.2, 2.15) -- (2.2,0.75) -- (0.8,0.75) -- (0.8,1.25);
\draw[thick, rounded corners=10pt] (0.2,1.65) -- (0.2,2.8) --
(2.65, 2.8);
\draw[thick, rounded corners=10pt] (0.2,1.25) -- (0.2,0.2) --
(2.65, 0.2);
\draw[thick] (5,1.7) to[out=-20, in=-160] (6,1.7);
\draw[thick] (5,1.3) to[out=-20, in=-160] (6,1.3);
\draw[thick] (5.2,1.65) -- (5.2,1.25); 
\draw[thick] (5.8,1.65) -- (5.8,1.25); 
\draw[thick] (7,1.7) to[out=-20, in=-160] (8,1.7);
\draw[thick] (7,1.3) to[out=-20, in=-160] (8,1.3);
\draw[thick] (7,1.44) -- (7.2,1.25); 
\draw[thick] (7.8,1.65) -- (8.0,1.46); 
\draw[dashed] (7,1.44) -- (8.0,1.46); 
\draw[thick] (7.2,1.65) -- (7.8,1.25); 
\draw[thick] (4.7,3) to[out=-70, in=70] (4.7,2);
\draw[thick] (4.3,3) to[out=-70, in=70] (4.3,2);
\draw[thick] (4.7,1) to[out=-70, in=70] (4.7,0);
\draw[thick] (4.3,1) to[out=-70, in=70] (4.3,0);
\draw[thick] (3.2,3) to[out=-80, in=80] (3.2,0);
\draw[thick] (2.6,3) to[out=-80, in=80] (2.6,0);
\draw[thick] (2.9,2.9) -- (3.0,2.8);
\draw[thick] (2.9,2.7) -- (3.0,2.8);
\draw[thick] (2.9,0.1) -- (3.0,0.2);
\draw[thick] (2.9,0.3) -- (3.0,0.2);
\draw[thick] (2.65,1.5) -- (2.75,1.4);
\draw[thick] (2.85,1.5) -- (2.75,1.4);
\draw[thick] (2.53,0.17) -- (2.63,0.07);
\draw[thick] (2.73,0.17) -- (2.63,0.07);
\draw[thick] (3.27,1.5) -- (3.37,1.6);
\draw[thick] (3.47,1.5) -- (3.37,1.6);
\draw[thick] (3.11,2.81) -- (3.21,2.91);
\draw[thick] (3.31,2.81) -- (3.21,2.91);
\draw[thick] (2.65,2.8) -- (4.77,2.8);
\draw[thick] (4.35,2.2) -- (4.76,2.2);
\draw[thick] (4.35,0.8) -- (4.54,1.0);
\draw[thick] (2.65,0.2) -- (4.35,0.2);
\draw[thick] (4.35,0.2) -- (4.76,0.8);
\draw[thick] (4.56,0.0) -- (4.76,0.2);
\draw[dashed] (4.56,0.0) -- (4.54,1.0);
\draw[thick, rounded corners=10pt] (4.35,0.8) -- 
(3.68,0.8) -- (3.68,2.2) -- (4.35,2.2);
\draw[thick, rounded corners=10pt] (4.76,0.2) -- (7.8,0.2) --
(7.8, 1.25);
\draw[thick, rounded corners=10pt] (4.76,2.8) -- (7.8,2.8) --
(7.8, 1.65);
\draw[thick, rounded corners=10pt] (4.76,0.8) -- (5.2,0.8) --
(5.2, 1.25);
\draw[thick, rounded corners=10pt] (4.76,2.2) -- (5.2,2.2) --
(5.2, 1.65);
\draw[thick, rounded corners=10pt] (5.8,1.25) -- 
(5.8,0.8) -- (7.2,0.8) --(7.2, 1.25);
\draw[thick, rounded corners=10pt] (5.8,1.65) -- 
(5.8,2.2) -- (7.2,2.2) --(7.2, 1.65);
\fill 
(0.2,1.65) circle (1.4pt) (0.8,1.65) circle (1.4pt)
(0.2,1.25) circle (1.4pt) (0.8,1.25) circle (1.4pt)
(5.2,1.65) circle (1.4pt) (5.8,1.65) circle (1.4pt)
(5.2,1.25) circle (1.4pt) (5.8,1.25) circle (1.4pt)
(7.2,1.65) circle (1.4pt) (7.8,1.65) circle (1.4pt)
(7.2,1.25) circle (1.4pt) (7.8,1.25) circle (1.4pt);
\fill 
(4.76,0.2) circle (1.4pt) (4.76,2.8) circle (1.4pt)
(4.35,0.2) circle (1.4pt) (4.35,2.8) circle (1.4pt)
(3.23,0.2) circle (1.4pt) (3.23,2.8) circle (1.4pt)
(2.63,0.2) circle (1.4pt) (2.63,2.8) circle (1.4pt);
\fill
(4.76,0.8) circle (1.4pt) (4.76,2.2) circle (1.4pt)
(4.35,0.8) circle (1.4pt) (4.35,2.2) circle (1.4pt);
\end{scope}
\end{tikzpicture}\par\medskip\noindent
\end{center}
\label{fig:PD0}
\caption{a pants decomposition and one of its associated cell decompositions}
\end{figure}

{We have two kinds of $2$-cells: hexagons and squares. Each pair of pants includes $2$ hexagons, while each  $\delta_i$ involves $2$ squares.}
The cell decomposition $\SP$ is not canonical, since there remain ambiguities coming from the Dehn twists along the simple closed curves $\delta_i$'s. By choosing one of the cell decompositions, 
we obtain a branch of the twisting parameter $\tau_i$ for each $\delta_i$, $1 \leq i \leq 3g-3$. So it enables us to define the parameter $\tau_i$ without using the simple connectivity 
of the Teichm\"uller space $\Tg$. (Compare \cite{IT}[\S3.2.4, p.63].)
\par
In general, any $1$-cocycle on a cell complex admits the ambiguity coming from any gauge transformation on its vertices. In the cell decomposition $\SP$, each vertex is included in some pair of pants. So, if we fix a $1$-cocycle on each pair of pants, then there remains no ambiguity for the $1$-cocycle. We introduce a normalization condition on a $1$-cocycle for any hyperbolic pair of pants to obtain a unique $1$-cocycle on $\SP$ for each point of the Teichm\"uller space $\Tg$.  We call both of them the {\it standard cocycles}.\par
The purpose of the present paper is to show how our standard cocycle illuminates some properties of the Teichm\"uller space $\Tg$.
Through the standard cocycle, the holonomy of any hyperbolic structure on the surface $\Sg$ is explicitly given in terms of the Fenchel-Nielsen coordinates. Moreover it leads us a topological proof of Wolpert's formula \eqref{eq:Wformula}. 
Here we follow the description of the form $\WP$ given in the seminal papers \cite{G1,G2} by Goldman. 
{As was shown by Goldman \cite{G1}, the Weil-Petersson symplectic form $\WP$ in Wolpert \cite{W2} equals the Atiyah-Bott-Goldman symplectic form $\ABGtr$ defined as a cup product on the twisted cohomology group with the trace form on $\sltwo$ up to some multiplicative constant. It is proved by some arguments due to \cite{A, H, W1}, as was stated by \cite[p.216]{G1}.
More precisely, it is essentially due to Wolpert's argument \cite[\S4]{W1} based on Hejhal's result \cite{H}. But all the arguments are implicit, and we find no explicit references for this proof. Hence, in \S6, we trace all the arguments carefully to determine the multiplicative constant: $\WP = 2\ABGtr$, which matches our computation with Wolpert's result \ref{eq:Wformula}.}\par
In our computation the form $\WP$ localizes near the simple closed curves $\delta_i$'s. Here it should be remarked that this is a special phenomenon which holds only for $\PSL$. As was already pointed out by Goldman \cite{G3}, Wienhard-Zhang \cite{WZ}, Sun-{Wienhard}-Zhang \cite{SWZ} {and Sun-Zhang \cite{SZ}}, if $n \geq 3$, the Atiyah-Bott-Goldman symplectic form for $SL_n(\bR)$ requires not only the information along the simple closed curves $\{\delta_i\}^{3g-3}_{i=1}$, but also some information inside of each pair of pants. {In fact, H. C. Kim \cite{Kim99} establishes Wolpert's formula on the moduli of real projective structures ($n=3$) involved with `internal parameters'. As the referee points out, squares in our decomposition have some similarity to another square bounded by blue edges and green dotted bridges in \cite[Figure 16]{SZ}. In order to generalize our result to $SL_n(\bR)$ for $n \geq 3$, it seems required to find some new insights about hexagons. }
\par
The present paper is organized as follows. 
In \S1, we introduce the notion of the standard cocycle for any hyperbolic structure on a pair of pants $P = \Sigma_{0,3}$ with geodesic boundary, and prove its existence and uniqueness (Lemma \ref{prop:std1}). The structure is determined by the geodesic length $l_k$ of each boundary component 
$\pa_kP$, $k \in \bZ/3= \{0,1,2\}$. So we compute the cocycle in terms of the lengths $l_0$, $l_1$ and $l_2$ explicitly (Theorem \ref{thm:A_k}). Our computation in \S1 is quite classical, and follows those given by Keen\cite{Keen1, Keen2}.  In \S2, we introduce a cell decomposition $\SP$ of $\Sg$ associated with the pants decomposition $\cP$. 
This enables us to define the twisting parameter $\tau_i: \Tg \to \bR$ without using the fact $\pi_1(\Tg) = 1$. More precisely, 
planting each pair of pants with the standard cocycle, 
we construct a groupoid $1$-cocycle on $\SP$ representing each point $[\rho] \in \Tg$, which we also call the {\it standard cocycle} for $[\rho]$ on $\SP$. 
Then the value $\tau_i([\rho])$ equals $(-2)$ times the logarithm of some coefficient of the groupoid cocycle.  We compute explicitly the first variation of the standard cocycle in \S3.
Using the computation in \S3, we compute explicitly the symplectic form $\WP$ described as a cup product by Goldman \cite{G1,G2} to obtain a topological proof of Wolpert's formula \eqref{eq:Wformula}in \S4. 
{Here we introduce explicit cellular approximations of the diagonal map on two kinds of $2$-cells as in Figures 5 and 6.}
{In \S5}, we discuss lifts of the standard cocycle to $\SL$, namely, the standard cocycle for a spin hyperbolic surface. {In the last section \S6, we prove $\WP = 2\ABGtr$ following Wolpert's argument \cite[\S4]{W1} based on Hejhal's result \cite{H}.}
\par
\medskip
\noindent
{\bf The fundamental groupoid and holonomy homomorphisms}: In this paper, the concatenation of two paths $\gamma_1\gamma_2$ on a space $X$ means that one traverses $\gamma_1$ first, transfers at $\gamma_1(1) = \gamma_2(0)$ and traverses $\gamma_2$. In other words, we consider the dual category $\Pi X\op$ of the fundamental groupoid of $X$. The set of objects in $\Pi X\op$ is $X$, and $\Pi X\op(x_1, x_0) := [([0,1], 0, 1), (X, x_0, x_1)]$, 
the homotopy set of paths running from $x_0$ to $x_1$. The composite is induced by the concatenation stated above. 
As usual, the group $\PSL$ acts on the upper half plane $\bH$ from the left. 
In this paper, we define the holonomy $\rho: \pi_1(X, *) \to \mathrm{Aut}(\widetilde{X}/X)$ of the universal covering $\pi: (\widetilde{X}, \tilde*) \to (X, *)$ by $\tilde \ell (1) = \rho([\ell])(\tilde*)$ for any loop $\ell: ([0,1], \{0,1\}) \to (X, *)$ and its lift $\tilde\ell: ([0,1], 0) \to (X, \tilde*)$. Then our concatenation convention makes the holonomy $\rho$ {\it a group homomorphism}. \par
\bigskip
\noindent
{\bf Acknowledgement}: The author sincerely thanks Tsukasa Ishibashi and Paul Norbury for helpful discussions. In particular, \S\ref{sec:spin} is inspired by discussions with P.N. Moreover he also thanks one of the audiences in his talk at the MFO workshop `Teichm\"uller Theory: Classical, Higher, Super and Quantum' (August 2023) who pointed out the ambiguity of the cell decomposition coming from Dehn twists. 
{Moreover he thanks the anonymous referee for his/her valuable comments to an earlier version.}
The present research is supported in part by the grants JSPS KAKENHI 19H01784
, 20H00115
, 22H01125
, 22H01120 
and the joint project between OIST (Hikami Unit) and University of Tokyo.
\par


\section{The standard cocycle on a pair of pants}

Let $P = \Sigma_{0,3}$ be a pair of pants, namely, a compact connected oriented surface of genus $0$ with $3$ boundary components. Fix a labeling of the boundary components, and denote each of them by $\pa_kP$, $k \in \bZ/3 = \{0,1,2\}$. We consider a cell decomposition of the pair of pants $P$ as in Figure 2. 

\begin{figure}[h]
\begin{center}
\begin{tikzpicture}[scale=0.75]
\tikzset{Process/.style={rectangle,  draw,  text centered, 
text width=37mm, minimum height=1.5cm}};
\coordinate (A) at (0,6);
\coordinate (B) at (14,6);
\coordinate (C) at (0,0);
\coordinate (D) at (14,0);
\coordinate (p) at (2,4);
\coordinate (q) at (6,4);
\coordinate (r) at (2,2);
\coordinate (s) at (6,2);
\coordinate (x) at (8,4);
\coordinate (y) at (12,4);
\coordinate (z) at (8,2);
\coordinate (w) at (12,2);
\coordinate (001) at (0,3);
\coordinate (010) at (2,3);
\coordinate (011) at (6,3);
\coordinate (020) at (8,3);
\coordinate (021) at (12,3);
\coordinate (000) at (14,3);
\fill (000) circle (2pt) (001) circle (2pt)
(010) circle (2pt) (011) circle (2pt)
(020) circle (2pt) (021) circle (2pt);
\draw[thick] (001) -- (010) node [midway] {$<$}
node [midway] [below] {$e^1_1$} ;
\draw[thick] (011) -- (020) node [midway] {$<$}
node [midway] [below] {$e^1_2$} ;
\draw[thick] (021) -- (000) node [midway] {$<$}
node [midway] [below] {$e^1_0$} ;
\node at (001)  [left] {$e^0_{01}$}; 
\node at (000) [right] {$e^0_{00}$}; 
\node at (011)  [left] {$e^0_{11}$}; 
\node at (010) [right] {$e^0_{10}$}; 
\node at (021)  [left] {$e^0_{21}$}; 
\node at (020) [right] {$e^0_{20}$}; 
\draw[thick, rounded corners=30pt] (001) -- (A) -- (B) node [midway] {$<$} node [midway] [below] {$e^1_{00}$} -- (000);
\draw[thick, rounded corners=30pt] (001) -- (C) -- (D) node [midway] {$>$} node [midway] [above] {$e^1_{01}$} -- (000);
\draw[thick, rounded corners=20pt] (010) -- (p) -- (q) node [midway]  {$>$} node [midway] [above] {$e^1_{10}$} -- (011);
\draw[thick, rounded corners=20pt] (010) -- (r) -- (s) node [midway]  {$<$} node [midway] [below] {$e^1_{11}$} -- (011);
\draw[thick, rounded corners=20pt] (020) -- (x) -- (y) node [midway] {$>$} node [midway] [above] {$e^1_{20}$} -- (021);
\draw[thick, rounded corners=20pt] (020) -- (z) -- (w) node [midway] {$<$} node [midway] [below] {$e^1_{21}$}  -- (021);
\node () at (7,4.5) {$e^2_+$};
\node () at (7,1.5) {$e^2_-$};
\node () at (4,3) {$\partial_1P$};
\node () at (10,3) {$\partial_{{2}}P$};
\node () at (15,2) {$\partial_0P$};
\end{tikzpicture}\par\medskip\noindent
\end{center}
\caption{Cell decomposition of $P = \Sigma_{0,3}$}
\end{figure}
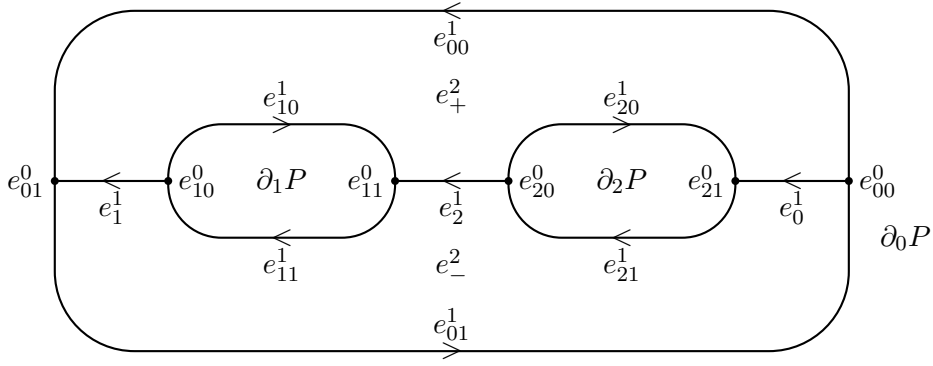
Here we have $6$ vertices(= $0$-cells), $9$ edges(= $1$-cells) and $2$ faces(= $2$ cells) in this decommposition. Later we call each of these $2$-cells $e^2_+$, $e^2_-$ a {\it hexagon}. 
The orientation of each $\pa_kP$ is induced by that of the surface $P$. 
We have $2$ vertices $e^0_{k0}$, $e^0_{k1}$ on each $\pa_kP$. 
The edge $e^1_{k0}$ is the subarc of $\pa_kP$ running from $e^0_{k0}$ to 
$e^0_{k1}$ along the orientation of $\pa_kP$, and $e^1_{k1}$ the subarc from $e^0_{k1}$ to $e^0_{k0}$. The edge $e^1_k$ is an arc from $e^0_{k,0}$ to $e^0_{k-1,1}$. When we endow $P$ with a hyperbolic structure with geodesic boundary, we can take each edge $e^1_k$ as the {\it seam}, i.e., the shortest geodesic connecting the geodesics $\pa_{k-1}P$ and $\pa_kP$. \par
The cell decomposition has a $\bZ/3$-symmetry given by changing the labeling of the boundary components cyclically. In fact, any of such a cyclic permutation can be realized by an orientation-preserving diffeomorphism of the surface $P$, which preserves also the cell decomposition. Let $P^{(0)}$ be the $0$-skeleton of the decomposition. 
In this section, we will  construct explicitly a $\bZ/3$-equivariant groupoid cocycle on the groupoid $\Pi P\op\vert_{P^{(0)}}$, the dual category of the fundamental groupoid of $P$ whose object set is restricted to the subset $P^{(0)}$, associated with a hyperbolic structure on $P$ making each of $\pa_k P$ a geodesic. 
The edges of $P$ generate the groupoid $\Pi P\op\vert_{P^{(0)}}$, while the $2$-cells $e^2_{\pm}$ provide its defining relations. In particular, any $1$-cocycle on $\Pi P\op\vert_{P^{(0)}}$ is determined uniquely by its values on the edges, and 
vanishes on the relations coming from the faces. 
\par

We begin by recalling some elementary facts on $\PSL$. Throughout this paper, we denote 
$$
\bD(h) := \pm \twomatrix{h}{0}{0}{h\inv} \in \PSL.
$$
for a positive real number $h > 0$. 
\begin{lem}\label{lem:ABBA}
For $\la > 1$ and $A, B \in \PSL$, we have the following.\par
{\rm (1)} If $A\bD(\la) = \bD(\la)A$, then we have $A = \bD(a)$ for some $a > 0$. \par
{\rm (2)} If $AB = BA = \bD(\la)$, then we have $A = \bD(a)$ and 
$B = \bD(\la/a)$ for some $a > 0$. \par
{\rm (3)} If $A \bD(\la)B\inv = B\bD(\la)A\inv = \bD(\la)$, then we have $A = B = \bD(a)$ for some $a > 0$. 
\end{lem}
\begin{proof} (1) Denote $A = \pm\twomatrix{a}{b}{c}{d}$. Then 
$\twomatrix{\la a}{\la\inv b}{\la c}{\la\inv d} = \pm \twomatrix{a}{b}{c}{d}
\twomatrix{\la}{0}{0}{\la\inv}
= A\bD(\la) = \bD(\la)A= \pm 
\twomatrix{\la}{0}{0}{\la\inv}\twomatrix{a}{b}{c}{d}
= \pm\twomatrix{\la a}{\la b}{\la\inv c}{\la\inv d}$, 
so that $\la^2|b| = |b|$ and $\la^2|c| = |c|$. 
Since $\la^2 \neq 1$, we have $b = c = 0$. 
From $\det A = 1$, we have $d = a\inv$. We may choose $a > 0$ 
by changing the sign of $a$ if necessary. \par
(2) Since $AB = BA$, we have $A(BA) = (AB)A = (BA)A$, while 
$BA = \bD(\la)$. Hence, by (1), we have $A = \bD(a)$ for some $a > 0$. 
We conclude $B = \bD(\la/a)$ from $BA = \bD(\la)$. \par
(3) We have $A\bD(\la^2)A\inv 
= A\bD(\la)B\inv B\bD(\la)A\inv
= \bD(\la^2)$ with $\la^2 > 1$. Hence, by (1), 
we have $A = \bD(a)$ for some $a > 0$. 
Hence we obtain $B = \bD(\la)\bD(a)\bD(\la\inv) = \bD(a)$. 
\end{proof}

\par
Now we fix a hyperbolic structure on $P$ which makes each of $\pa_k P$ a geodesic. Such a hyperbolic structure is uniquely determined by the length of each $\pa_kP$, $k \in \bZ/3$, which we denote by $l_k$. 
Its holonomy is given by a groupoid cocycle $\rho: \Pi P\op\vert_{P^{(0)}} \to \PSL$, which is unique up to conjugate by a $0$-cochain, namely a function $b: P^{(0)} \to \PSL$. We denote by $\rho^{b}$ 
the conjugate of $\rho$ by a function $b: P^{(0)} \to \PSL$. For any 
edge $e$ running from a vertex $v_0$ to $v_1$, we have 
\begin{equation}\label{eq:brho}
(\rho^{b})(e) := b(v_0)\inv\rho(e)b(v_1).
\end{equation}
We will prove there is a unique cocycle $\rho$ which represents the hyperbolic structure and satisfies certain conditions
(Proposition \ref{prop:std1}). We will call it the {\it standard cocycle} associated with the hyperbolic structure. 
In our proof, we modify cocycles gradually, but we denote all of them by the same symbol $\rho$. \par
Let $\la_k( >1)$ and ${\la_k}\inv$ be (the absolute values of) the eigenvalues of the holonomy along the boundary component $\pa_kP$, which satisfies 
$\la_k = \exp(l_k/2)$ or equivalently $l_k = 2\log\la_k$. 
\begin{lem}\label{lem:diagonal} By changing the cocycle $\rho$ by some conjugate, we have 
\begin{equation}\label{eq:norm1}
\forall k \in \bZ/3, \quad
\rho(e^1_{k0}) = \rho(e^1_{k1}) = \bD({\la_k}^{1/2}).
\end{equation}
Moreover, the conjugate $\rho^{b}
$ by a function $b: P^{(0)}\to \PSL$ 
satisfies the condition \eqref{eq:norm1}, if and only if we have 
$b(e^0_{k0}) = b(e^0_{k1})= \bD(t_k)$ for some $t_k > 0$. 
\end{lem}
\begin{proof} We have $\rho(e^1_{k0})\rho(e^1_{k1}) = {S_k} \bD(\la_k){S_k}\inv$ and 
$\rho(e^1_{k1})\rho(e^1_{k0}) = {T_k}\bD(\la_k){T_k}\inv $ 
for some $S_k, T_k \in \PSL$. If we take its conjugate by the function $b: P^{(0)} \to \PSL$ defined by $b(e^0_{k0}) = S_k$ and $b(e^0_{k1}) = T_k$, we have 
$$
\rho(e^1_{k1})\rho(e^1_{k0}) = \rho(e^1_{k0})\rho(e^1_{k1}) 
= \bD(\la_k).
$$
Hence, by Lemma \ref{lem:ABBA} (2), we have 
$$
\rho(e^1_{k0}) =\bD(\mu_k), \quad
\rho(e^1_{k1}) = \bD(\la_k/\mu_k)
$$
for some $\mu_k > 0$. If we take its conjugate by the function $b: P^{(0)}\to \PSL$ defined by $b(e^0_{k0}) = \bD(1) = \pm E$ and $b(e^0_{k1}) = \bD({\la_k}^{1/2}/\mu_k)$, we have 
$$
\rho(e^1_{k0}) = \rho(e^1_{k1}) = \bD({\la_k}^{1/2}),
$$
which proves the first half of the lemma.
The second half follows from Lemma \ref{lem:ABBA} (3). 
\end{proof}

In order to re-define the Fenchel-Nielsen coordinates 
by using a groupoid cocycle satisfying the condition \eqref{eq:norm1} in \S\ref{sec:FN},
we recall another elementary fact on the hyperbolic plane. 
\begin{lem}\label{lem:axis1} Let $\la >1$ and 
$A = \pm \twomatrix{a}{b}{c}{d}\in \PSL$.
Then the attracting fixed point of the hyperbolic transformation 
$A\bD(\la)A\inv$ is $a/c$, and the repelling fixed point $b/d$. 
Moreover, if the axis of $A\bD(\la)A\inv$ does not intersect with the imaginary axis, then the point on the imaginary axis nearest to that of $A\bD(\la)A\inv$
is $\sqrt{ab/cd}\sqrt{-1}$. 
\end{lem}
\begin{proof}
The vectors $\twovector{a}{c}$ and 
$\twovector{b}{d}$ are eigenvectors of $A\bD(\la)A\inv$
for the eigenvalues $\pm\la$ and $\pm\la\inv$, respectively. This implies the first half of the lemma. 
\par
To prove the second half, it suffices to compute the shortest geodesic 
between the imaginary axis and the circular arcs perpendicular to the real axis
at the points $q < p < 0$. The latter is a segment of the Euclidean circle $C_1$ centered at $(p+q)/2$ with radius $(p-q)/2$. 
Let $R\sqrt{-1}$, $R > 0$, be the intersection point 
of the shortest geodesic and the imaginary axis.
Then the shortest geodesic is a segment of the Euclidean circle $C_2$ centered at the origin $0$ with radius $R$.
Then the circles $C_1$ and $C_2$ are perpendicular to each other.
Hence we have $((p-q)/2)^2+R^2 = ((p+q)/2)^2$, so that $R = \sqrt{pq}$. 
The same result holds also for the case $0 < p < q$, 
since $R = \sqrt{pq}$ is symmetric in $p$ and $q$.
The second half follows from this computation and the first half.
\end{proof}

\par
Now let $\rho$ be a groupoid cocycle which is a holonomy of the fixed hyperbolic structure and satisfies the condition \eqref{eq:norm1}.
For the rest of this section, we denote 
$$
\rho(e^1_k) = A_k = \pm \twomatrix{a_k}{b_k}{c_k}{d_k}, \quad
a_kd_k-b_kc_k = 1. 
$$
Here we recall some classical consideration on the axes of $\rho(\gamma_k)$'s,
which was already discussed by Keen 
\cite[pp.415-417]{Keen2}. 
By gluing a half $\{z\in \bC; \,\, \Im z> 0, \, \Re z\geq 0\}/\langle \bD(\la_k)\rangle$ of a hyperbolic annulus along each geodesic boundary $\pa_kP$, 
we obtain a complete hyperbolic surface $\hat P$, whose universal covering is the upper half plane $\bH$. We take a universal covering map $\varpi: \bH \to \hat P$ such that the deck transformation of $\pi_1(P, e^0_{00}) = \pi_1(\hat P, e^0_{00})$ is given by $\rho\vert_{\pi_1(P, e^0_{00})}$. 
For each $k \in \bZ/3$, we denote the attracting fixed point and 
the repelling fixed point of $\rho(\gamma_k)$ 
by $\infty_k$ and $0_k$, respectively. 
For example, $0_0 = 0$, $\infty_0 = \infty$, $0_2 = b_0/d_0$ and 
$\infty_2 = a_0/c_0$. 

\begin{lem} For any $k \in \bZ/3$, we have 
\begin{equation}\label{eq:ab_cd}
b_k/d_k < a_k/c_k < 0. 
\end{equation}
Moreover we have
\begin{equation}\label{eq:ad_bc}
a_kb_k/c_kd_k > 0, \quad
b_kc_k < 0, \quad 
{a_kb_k > 0}
\quad\text{and}\quad
c_kd_k > 0.
\end{equation}
\end{lem}
\begin{proof} The axis of $\rho(\gamma_2)$ does not intersect with 
that of $\gamma(\gamma_0)$, the imaginary axis. Considering the orientation of these axes, we conclude the axis of $\rho(\gamma_2)$ is located 
on the left-hand side of the imaginary axis. This means $0_2 < \infty_2 < 0$, i.e., $b_0/d_0 < a_0/c_0 < 0$. By the $\bZ/3$-symmetry, 
we obtain \eqref{eq:ab_cd}
for any $k \in \bZ/3$.\par
The inequality \eqref{eq:ab_cd} implies 
$a_kb_k/c_kd_k > 0$ and $|a_kd_k| < |b_kc_k|$.
Assume $b_kc_k > 0$. Then $|a_kd_k| < |b_kc_k| = b_kc_k = a_kd_k - 1
\leq |a_kd_k|-1$, which is a contradiction. Hence we have $b_kc_k < 0$. 
{By} \eqref{eq:ab_cd}, 
we obtain {$a_kb_k > 0$ and} $c_kd_k>0$. This completes the proof. 
\end{proof}

Let $\gamma_0$, $\gamma_1$ and $\gamma_2$ be elements of the fundamental group $\pi_1(P, e^0_{00})$
defined by 
$$
\gamma_0 := e^1_{00}e^1_{01}, \quad
\gamma_1 := (e^1_0(e^1_{20})\inv e^1_2) 
e^1_{11}e^1_{10}
(e^1_0(e^1_{20})\inv e^1_2)\inv \quad\text{and}\quad
\gamma_2 := e^1_0 e^1_{21}e^1_{20}(e^1_0)\inv,
$$
respectively. Then the relation
\begin{equation}\label{eq:pi_rel}
\gamma_2\gamma_1\gamma_0 = 1 \in \pi_1(P, e^0_{00})
\end{equation}
holds. 
The group $\pi_1(P, e^0_{00})$ is freely generated by 
$\{\gamma_0, \gamma_2\}$. We have 
$\rho(\gamma_0) = \bD(\la_0)$ and 
$\rho(\gamma_2) = {A_0} \bD(\la_2){A_0}\inv$. \par

Since ${\gamma_1}\inv = \gamma_0\gamma_2$ from \eqref{eq:pi_rel}, 
we have 
$$
\aligned
& \rho(\gamma_1)\inv = \rho(\gamma_0)\rho(\gamma_2) 
= \bD(\la_0){A_0}\bD(\la_2){A_0}\inv \\
&= \pm\twomatrix{\la_0}{0}{0}{{\la_0}\inv}
\twomatrix{a_0}{b_0}{c_0}{d_0}
\twomatrix{\la_2}{0}{0}{{\la_2}\inv}
\twomatrix{d_0}{-b_0}{-c_0}{a_0}
\\
&= \pm \twomatrix{\la_0a_0}{\la_0b_0}{{\la_0}\inv c_0}{{\la_0}\inv d_0}\twomatrix{\la_2d_0}{-\la_2b_0}{-{\la_2}\inv c_0}{{\la_2}\inv a_0}\\
&= \pm\twomatrix{{\la_0}{\la_2}a_0d_0 - {\la_0}{\la_2}\inv b_0c_0}
{{\la_0}{\la_2}\inv a_0b_0 - {\la_0}{\la_2} a_0b_0}
{{\la_0}\inv{\la_2}c_0d_0 - {\la_0}\inv{\la_2}\inv c_0d_0}
{{\la_0}\inv{\la_2}\inv a_0d_0 - {\la_0}\inv {\la_2} b_0c_0}.
\endaligned
$$
We denote by $\twomatrix{p_1}{q_1}{r_1}{s_1}$ the right-hand side of the above equation deleting the sign $\pm$. For example, 
$r_1 = {\la_0}\inv({\la_2}- {\la_2}\inv) c_0d_0$. 
Then we have the following lemma proved by Keen \cite{Keen1}. 
In order to make this paper self-contained, we give a proof of the lemma. Moreover it will be proved again in \eqref{eq:spinP}, \S{5}.3. 
\begin{lem}[\cite{Keen1}, Lemma 1, p.210]\label{lem:Keen}
$$
p_1+s_1 < 0.
$$
\end{lem}
\begin{proof}
We have $r_1 > 0$, since $\la_0>0$, $({\la_2}- {\la_2}\inv)> 0$ and $c_0d_0 > 0$ from \eqref{eq:ad_bc}.
The fixed points of $\rho(\gamma_1)\inv$ are $(p_1-s_1 \pm\sqrt{(p_1+s_1)^2-4})/2r_1$. Since $r_1 > 0$, we have 
$$
(p_1-s_1 -\sqrt{(p_1+s_1)^2-4})/2r_1
< (p_1-s_1 +\sqrt{(p_1+s_1)^2-4})/2r_1
$$
which correspond to the eigenvalues 
$(p_1+s_1 -\sqrt{(p_1+s_1)^2-4})/2$ and 
$(p_1+s_1 +\sqrt{(p_1+s_1)^2-4})/2$, respectively.
The orientation of the axis of $\rho(\gamma_1)\inv$ is opposite to that of the upper sides of the axis. Hence $(p_1-s_1 -\sqrt{(p_1+s_1)^2-4})/2r_1$ is attracting, and so 
$|(p_1+s_1 -\sqrt{(p_1+s_1)^2-4})/2| > 1$. Hence we have $p_1+s_1 < 0$, as was to be shown.
\end{proof}
By this lemma we have 
$$
\aligned
-\la_1 - {\la_1}\inv &= p_1+ s_1\\
&= 
({\la_0}{\la_2}+ {\la_0}\inv{\la_2}\inv)a_0d_0 
 - ( {\la_0}{\la_2}\inv+{\la_0}\inv {\la_2}) b_0c_0\\
 &= ({\la_0}{\la_2}+ {\la_0}\inv{\la_2}\inv)
 + (\la_0-{\la_0}\inv)(\la_2-{\la_2}\inv)b_0c_0,
\endaligned
$$
where we used $a_0d_0 = 1+b_0c_0$. 
From \eqref{eq:ad_bc} and $\la_k = \exp(l_k/2)$, we have 
$$
|b_0c_0| = -b_0c_0 = \frac{\cosh(l_1/2) + \cosh((l_2+l_0)/2)}{2\sinh(l_2/2)\sinh(l_0/2)}.
$$
By the $\bZ/3$-symmetry and the well-known formulae 
$$
\aligned
& \cosh t + \cosh s = 2\cosh(\frac{t+s}2)\cosh(\frac{t-s}2),\\
& 2\sinh t \sinh s = \cosh (t+s) - \cosh(t-s),\\
\endaligned
$$
we obtain 
\begin{equation}\label{eq:bkck}
\aligned
|b_kc_k| &= \frac{\cosh(l_{k+1}/2) + \cosh((l_{k-1}+l_k)/2)}{2\sinh(l_{k-1}/2)\sinh(l_k/2)}\\
&= 
\frac{\cosh((l_0+l_1+l_2)/4)\cosh((l_0+l_1+l_2)/4 - l_{k+1}/2)}{\sinh(l_{k-1}/2)\sinh(l_k/2)}
\endaligned
\end{equation}
and
\begin{equation}\label{eq:bkck-1}
\aligned
&- a_kd_k = |b_kc_k|-1 \\
&= \frac{\cosh(l_{k+1}/2) + \cosh((l_{k-1}+l_k)/2)
- \cosh((l_{k-1}+l_k)/2) + \cosh((l_{k-1}-l_k)/2)}
{2\sinh(l_{k-1}/2)\sinh(l_k/2)}\\
&= 
\frac{\cosh((l_0+l_1+l_2)/4 - l_{k-1}/2)\cosh((l_0+l_1+l_2)/4 - l_{k}/2)}{\sinh(l_{k-1}/2)\sinh(l_k/2)} \quad(> 0)
\endaligned
\end{equation}
for any $k \in \bZ/3$.

In general, we call a matrix $A = \pm\twomatrix{a}{b}{c}{d} \in \PSL$ {\it normalized} if $ab/cd = 1$. 

\begin{lem}\label{lem:norm0}
Let $A = \pm\twomatrix{a}{b}{c}{d} \in \PSL$ satisfy $ab/cd > 0$, and let
$t_0, t_1$ be positive real numbers. Then $\bD(t_0)\inv A\bD(t_1)$ is normalized if and only if $t_0 = (ab/cd)^{1/4}$. 
\end{lem}
\begin{proof} If we denote $\pm\twomatrix{a'}{b'}{c'}{d'} = \bD(t_0)\inv A\bD(t_1)$,  then 
$$
\pm\twomatrix{a'}{b'}{c'}{d'}
= \pm\twomatrix{{t_0}\inv}{0}{0}{{t_0}}\twomatrix{a}{b}{c}{d}\twomatrix{t_1}{0}{0}{{t_1}\inv}
= \pm \twomatrix{{t_0}\inv{t_1}a}{{t_0}\inv{t_1}\inv b}{t_0t_1c}{t_0{t_1}\inv d}.
$$
and so $a'b'/c'd' = {t_0}^{-4}ab/cd$. This proves the lemma. 
\end{proof}

Now we remark $a_bb_k/c_kd_k > 0$ by \eqref{eq:ad_bc}.
\begin{prop}\label{prop:std1}
There is a unique groupoid cocycle $\rho: \Pi P\op\vert_{P^{(0)}} \to \PSL$ 
which is a holonomy of the hyperbolic structure and satisfies the conditions 
\eqref{eq:norm1} and 
\begin{equation}\label{eq:norm2}
\forall k \in \bZ/3, \quad
a_kb_k/c_kd_k = 1.
\end{equation}
We call it the {\it standard cocycle} of the hyperbolic structure on $P$.
\end{prop}
\begin{proof}
The condition \eqref{eq:norm2} means that $\rho(e^1_k)$ is normalized
for each $k \in \bZ/3$. 
From Lemma \ref{lem:diagonal}, a cocycle $\rho$ satisfying the condition \eqref{eq:norm1} is unique up to conjugate by a function $b: P^{(0)} \to \PSL$ such that $b(e^0_{k0}) = b(e^0_{k1})$ is a diagonal matrix for each $k \in \bZ/3$. So one can denote $b(e^0_{k0}) = b(e^0_{k1}) = \bD(s_k)$, $s_k > 0$. 
Then we have $\rho^{b}
(e^1_k) = \bD(s_k)\inv A_k\bD(s_{k-1})$.  Hence, by 
Lemma \ref{lem:norm0}, $\rho^{b}
(e^1_k)$ is normalized if and only if 
$s_k = (a_kb_k/c_kd_k)^{1/4}$. 
This proves the lemma.
\end{proof}

The hyperbolic structure is uniquely determined by the geodesic lengths $l_k$'s of the boundary components. Hence one can describe 
the standard cocycle $\rho$, or equivalently each $\rho(e^1_k) = A_k
= \pm\twomatrix{a_k}{b_k}{c_k}{d_k}$, 
$k \in \bZ/3$, in terms the lengths $l_0$, $l_1$ and $l_2$ explicitly. 
\begin{thm}\label{thm:A_k} We have
\begin{equation}\label{eq:A_kA_k}
\rho(e^1_k) = A_k =  \pm\twomatrix{\sqrt{|b_kc_k|-1}}{|b_kc_k|^{1/2}}{{-}|b_kc_k|^{1/2}}{-\sqrt{|b_kc_k|-1}}.
\end{equation}
Here $|b_kc_k| = -b_kc_k$ and $|b_kc_k|-1(> 0)$ are given by \eqref{eq:bkck}
and \eqref{eq:bkck-1} respectively. In particular, we obtain ${A_k}^2 = \pm I$. 
\end{thm}
\begin{proof} First we describe $A_0$ in terms of $|b_0c_0|$ and $b_0$. 
We have 
$$
c_0 = -|b_0c_0|/b_0
$$ 
since $b_0c_0 < 0$. From \eqref{eq:bkck-1} $0 > 1 - |b_0c_0| = 
1+b_0c_0 = a_0d_0$. 
Since $a_0b_0= c_0d_0$, we have 
$|a_0b_0|^2 = |a_0b_0c_0d_0|
= |a_0d_0||b_0c_0| = |b_0c_0|(|b_0c_0|-1)$, so that
$$
a_0 = |a_0b_0|/b_0 = \sqrt{|b_0c_0|(|b_0c_0|-1)}/b_0
$$
by {\eqref{eq:ad_bc}}. Again by $a_0d_0 = -(|b_0c_0| -1)$, we have 
$$
d_0 = -(|b_0c_0|-1)/a_0 = {-}b_0\sqrt{(|b_0c_0|-1)/|b_0c_0|}.
$$
Consequently we obtain
$$
\aligned
A_0 &= \pm\twomatrix{a_0}{b_0}{c_0}{d_0}
= \pm\twomatrix{\sqrt{|b_0c_0|(|b_0c_0|-1)}/b_0}{b_0}{-|b_0c_0|/b_0}{{-}b_0\sqrt{(|b_0c_0|-1)/|b_0c_0|}}\\
&= \pm\twomatrix{\sqrt{|b_0c_0|(|b_0c_0|-1)}/|b_0|}{|b_0|}{-|b_0c_0|/|b_0|}{{-}|b_0|\sqrt{(|b_0c_0|-1)/|b_0c_0|}}\\
&= \pm\twomatrix{\sqrt{|b_0c_0|-1}}{|b_0c_0|^{1/2}}{-|b_0c_0|^{1/2}}{{-}\sqrt{|b_0c_0|-1}}\twomatrix{|b_0|^{-1}|b_0c_0|^{1/2}}{0}{0}{|b_0||b_0c_0|^{-1/2}}.
\endaligned
$$
Hence, by the $\bZ/3$-symmetry, if we denote $\mu_k := |b_k|^{-1}|b_kc_k|^{1/2} > 0$ and
\begin{equation}\label{eq:B_k}
B_k := \pm \twomatrix{\sqrt{|b_kc_k|-1}}{|b_kc_k|^{1/2}}{{-}|b_kc_k|^{1/2}}{-\sqrt{|b_kc_k|-1}}, 
\end{equation}
then we have 
$$
A_k = B_k\bD(\mu_k)
$$
for any $k \in \bZ/3$. Here we remark ${B_k}^2 = \pm I$
and $B_k$ is normalized. 
Hence {$\bD(\mu_k){A_k}^{-1} = B_k\inv$ is also normalized, and so we have}
\begin{equation}\label{eq:mu_k^4}
{\mu_k} = \left(\frac{a_kc_k}{b_kd_k}\right)^{1/4}
{= \frac{|c_k|^{1/2}}{|b_k|^{1/2}}}.
\end{equation}
\par
Finally we identify the value $\mu_k$. 
We begin by looking at the shortest geodesic between $\pa_0P$ and $\pa_1P$. 
The point on the imaginary axis nearest to the axis of $\rho(\gamma_1)$ is ${\la_0}^{2n+1}\sqrt{-1}$ for some $n \in \bZ$, since that for
$\rho(\gamma_2)$ is $\sqrt{-1}$
{and each of these points is a lift of the foot of a seam of the hyperbolic pants $P$ on $\pa_0P$}. 
On the other hand, $\pa_1P$ is covered by the axis of $\rho(\gamma_1) = \bD({\la_0}^{1/2}){A_1}\inv \bD({\la_1})A_1\bD({\la_0}^{-1/2})$, where we have
$\bD({\la_0}^{1/2}){A_1}\inv  = \pm\twomatrix{{\la_0}^{1/2}}{0}{0}{{\la_0}^{-1/2}}\twomatrix{d_1}{-b_1}{-c_1}{a_1} = \pm\twomatrix{{\la_0}^{1/2}d_1}{{-}{\la_0}^{1/2}b_1}{-{\la_0}^{-1/2}c_1}{{\la_0}^{-1/2}a_1}$. Hence, by Lemma \ref{lem:axis1}, the point on the imaginary axis nearest to the axis of $\rho(\gamma_1)$ is 
$$
\sqrt{{\frac{{\la_0}b_1d_1}{{\la_0}\inv a_1c_1}}}\sqrt{-1}
= {\la_0}{\mu_1}^{-2}\sqrt{-1}
$$
{by \eqref{eq:mu_k^4}}.
Hence we obtain $\mu_1 = {\la_0}^{-n}$. 
By the $\bZ/3$-symmetry, we have $\mu_k = {\la_{k-1}}^{-n}$ for any $k \in \bZ/3$. Here we remark the integer $n$ is constant because of the $\bZ/3$-symmetry and the connectivity of the Teichm\"uller space of the surface $P$. 
Consequently we obtain 
\begin{equation}\label{A_kla_k-1}
A_k = B_k\bD({\la_{k-1}}^{-n}).
\end{equation}
Next we consider the rotation of $P$ by $\pi$ radian which permutes $\pa_1P$ and $\pa_2P$. The groupoid cocycle $\bar\rho$ defined by $\bar\rho(e^1_0) = {A_1}^{-1}$, 
$\bar\rho(e^1_1) = {A_0}^{-1}$, $\bar\rho(e^1_2) = {A_2}^{-1}$, 
{$\bar\rho(e^1_{00}) = \bar\rho(e^1_{01}) = \bD(\la_0)$,
$\bar\rho(e^1_{10}) = \bar\rho(e^1_{11}) = \bD(\la_2)$
and $\bar\rho(e^1_{20}) = \bar\rho(e^1_{21}) = \bD(\la_1)$}
is a holonomy of the the hyperbolic structure after the rotation, whose 
standard cocycle we denote by $\rho'$. Then, by Lemma \ref{lem:diagonal}, 
we have 
$\rho'(e^1_0) = \bD({t_0}^{-1}){A_1}^{-1}\bD(t_2) = \bD({t_0}^{-1}{\mu_1}^{-1})B_1\bD(t_2)$, 
$\rho'(e^1_1) = \bD({t_1}^{-1}){A_0}^{-1}\bD(t_0) = \bD({t_1}^{-1}{\mu_0}^{-1})B_0\bD(t_0)$ and
$\rho'(e^1_2) = \bD({t_2}^{-1}){A_2}^{-1}\bD(t_1) = \bD({t_2}^{-1}{\mu_2}^{-1})B_2\bD(t_1)$ for some $t_0, t_1, t_2 > 0$. Since each $\rho'(e^1_k)$ is normalized, 
we have $t_0 = {\mu_1}\inv$, $t_1 = {\mu_0}\inv$ and $t_2 = {\mu_2}\inv$ by Lemma\ref{lem:norm0}. 
Hence we have 
\begin{equation}\label{eq:rho'}
\aligned
& \rho'(e^1_0) = \bD({\mu_1}){A_1}^{-1}\bD({\mu_2}\inv) = B_1\bD({\la_1}^{n}),\\
& \rho'(e^1_1) = \bD({\mu_0}){A_0}^{-1}\bD({\mu_1}\inv) = B_0\bD({\la_0}^{n}),\\
& \rho'(e^1_2) = \bD({\mu_2}){A_2}^{-1}\bD({\mu_0}\inv) = B_2\bD({\la_2}^{n}). 
\endaligned
\end{equation}
On the other hand, we compute $\rho'(e^1_k)$ directly based on our computation stated above. We denote by $B'_k$ the corresponding $B_k$ to $\rho'$. Then, from \eqref{eq:bkck} and \eqref{eq:B_k}, we have $B'_0 = B_1$, 
$B'_1 = B_0$ and $B'_2 = B_2$. Hence, by \eqref{A_kla_k-1}, we obtain
$$
\aligned
& \rho'(e^1_0) = B'_0\bD({\la_1}^{-n}) = B_1\bD({\la_1}^{-n}), \\
& \rho'(e^1_1) = B'_1\bD({\la_0}^{-n}) = B_0\bD({\la_0}^{-n}), \\
& \rho'(e^1_2) = B'_2\bD({\la_2}^{-n}) = B_2\bD({\la_2}^{-n}).
\endaligned
$$
Comparing this with \eqref{eq:rho'}, we conclude $n=0$, and so $\mu_k = 1$ for
any $k \in \bZ/3$. This completes the proof.
\end{proof}
In particular, the point on the imaginary axis nearest to the axis of $\rho(\gamma_1)$ is ${\la_0}\sqrt{-1}$, which matches our geometric intuition. \par
We will discuss lifts of the standard cocycle to $SL_2(\bR)$ in \S\ref{sec:spin}.

\section{The Fenchel-Nielsen coordinates}\label{sec:FN}

Now we consider a closed oriented surface $\Sg$ of genus $g \geq 2$. 
Fix a pants decomposition $\cP = \{\delta_i\}^{3g-3}_{i=1}$ of the surface $\Sg$. {We suppose} each simple closed curve $\delta_i$ {is oriented}. 
Using the orientation, we regard it as a $C^\infty$ embedding $\delta_i: \bR/2\pi\bZ \to \Sg$, and take a closed annulus neighborhood $\mathfrak{a}_i$ 
such that $\mathfrak{a}_i \cap\mathfrak{a}_j = \emptyset$ for any $i \neq j$,
and choose an orientation-preserving diffeomorphism $\alpha_i: [-1,1] \times(\bR/2\pi\bZ) \overset\cong\to \mathfrak{a}_i$ for each $1 \leq i \leq 3g-3$. The choice of $\alpha_i$ has an ambiguity coming from a half integer twist along $\delta_i$. The complement $\Sg\setminus \bigcup^{3g-3}_{i=1}(\mathfrak{a}_i)^\circ$ of the interiors $(\mathfrak{a}_i)^\circ = \alpha_i(\mathopen]-1,1\mathclose[\times (\bR/2\pi\bZ))$ of the annulus neighborhood $\mathfrak{a}_i$ is diffeomorphic to the disjoint union of $2g-2$ pairs of pants. We label each of them, and denote $P_j$, $1 \leq j \leq 2g-2$. See Figures 1 and 3. 
\par
Labelling the boundary components of each pair of pants $P_j$, 
we plant $P_j$ with the pants decomposition 
$(P, \{e^0_{k0}, e^0_{k1}, e^1_k, e^1_{k0}, e^1_{k1}, e^2_+, e^2_-; \,\, k \in \bZ/3\})$ in the previous section to obtain a cell decomposition 
$(P_j, \{e^0_{(j)k0}, e^0_{(j)k1}, e^1_{(j)k}, e^1_{(j)k0}, e^1_{(j)k1}, e^2_{(j)+}, e^2_{(j)-}; \,\, k \in \bZ/3\})$. Here, using the map $\alpha_i$ we make the vertices $e^0_{(j)k0}$ and $e^0_{(j)k1}$ located at
$$
\aligned 
&e^0_{(j)k0} := \begin{cases}
\alpha_i(-1, 0), 
&\text{if $\pa_kP$ lies on the left-hand side of $\delta_i$,}\\
\alpha_i(+1, 0), 
&\text{if $\pa_kP$ lies on the right-hand side of $\delta_i$,}
\end{cases}\\
&e^0_{(j)k1} := \begin{cases}
\alpha_i(-1, \pi), 
&\text{if $\pa_kP$ lies on the left-hand side of $\delta_i$,}\\
\alpha_i(+1, \pi), 
&\text{if $\pa_kP$ lies on the right-hand side of $\delta_i$.}
\end{cases}\\
\endaligned
$$
As in Figure 3, we can paraphrase this definition. Suppose $\pa_{k'}P_{j'}$ lies on the left-hand side of $\delta_i$, and $\pa_{k''}P_{j''}$ on the right-hand side. Then the orientation of $\pa_{k'}P_{j'}$ is the same as $\delta_i$, and that of $\pa_{k''}P_{j''}$
is the reverse. Hence we have 
$\alpha_i(-1,0) = e^0_{(j')k'0}$, $\alpha_i(-1,\pi) = e^0_{(j')k'1}$, 
$\alpha_i(+1,0) = e^0_{(j'')k''0}$ and $\alpha_i(+1,\pi) = e^0_{(j'')k''1}$. 
We introduce an edge $e^1_{((i))0}$ (resp.\ $e^1_{((i))1}$) whose characteristic map is given by $t \in [-1, 1] \mapsto \alpha_i(t, 0)$ (resp.\ $t \in [-1,1] \mapsto \alpha_i(t, \pi)$), and a face $e^2_{((i))0}$ (resp.\ $e^2_{((i))1}$) 
whose characteristic map is given by $(t, s)\in [-1, 1]\times [0, \pi] \mapsto \alpha_i(t, s)$ (resp.\ $(t, s)\in [-1, 1]\times [0, \pi] \mapsto \alpha_i(t, s+\pi)$). 

\begin{figure}
\begin{center}
\begin{tikzpicture}[scale=1.4]
\tikzset{Process/.style={rectangle,  draw,  text centered, text width=37mm, minimum height=1.5cm}};
\coordinate (--) at (1,1); 
\coordinate (-0) at (1,2); 
\coordinate (-+) at (1,3); 
\coordinate (+-) at (3,1); 
\coordinate (+0) at (3,2); 
\coordinate (++) at (3,3); 
\fill (--) circle (1.4pt) (-0) circle (1.4pt) (-+) circle (1.4pt) 
 (+-) circle (1.4pt) (+0) circle (1.4pt) (++) circle (1.4pt) ;
\draw[thick] (0.8,1) -- (3.2,1);
\draw[thick] (2.3,3) -- (2.2,2.9); 
\draw[thick] (2.3,3) -- (2.2,3.1); 
\draw[thick] (0.8,2) -- (3.2,2);
\draw[thick] (2.3,2) -- (2.2,1.9); 
\draw[thick] (2.3,2) -- (2.2,2.1); 
\draw[thick] (0.8,3) -- (3.2,3);
\draw[thick] (2.3,1) -- (2.2,0.9); 
\draw[thick] (2.3,1) -- (2.2,1.1); 
\draw[thick] (1,0.5) -- (1,3.5);
\draw[thick] (1,1.6) -- (0.9,1.5); 
\draw[thick] (1,1.6) -- (1.1,1.5); 
\draw[thick] (1,2.6) -- (0.9,2.5); 
\draw[thick] (1,2.6) -- (1.1,2.5); 
\draw[thick, dashed] (2,0.5) -- (2,3.5);
\draw[thick] (2,3.5) -- (1.9,3.4); 
\draw[thick] (2,3.5) -- (2.1,3.4); 
\draw[thick] (3,0.5) -- (3,3.5);
\draw[thick] (3,1.4) -- (2.9,1.5); 
\draw[thick] (3,1.4) -- (3.1,1.5); 
\draw[thick] (3,2.4) -- (2.9,2.5); 
\draw[thick] (3,2.4) -- (3.1,2.5); 
\node[font=\tiny] (0) at (1.8,0.8) {$0$};
\node[font=\tiny] (pi) at (1.8,1.8) {$\pi$};
\node[font=\tiny] (2pi) at (1.8,2.8) {$2\pi$};
\node[font=\tiny] (delta) at (2.2,3.3) {$\delta_i$};
\node[font=\tiny] (10) at (2.4,2.7) {$e^1_{((i))0}$};
\node[font=\tiny] (11) at (2.4,1.7) {$e^1_{((i))1}$};
\node[font=\tiny] (100) at (2.4,0.7) {$e^1_{((i))0}$};
\node[font=\tiny] (21) at (1.6,2.3) {$e^2_{((i))1}$};
\node[font=\tiny] (20) at (1.6,1.3) {$e^2_{((i))1}$};
\node[font=\tiny] (minus) at (1.2,0.8) {$-1$};
\node[font=\tiny] (plus) at (2.8,0.8) {$+1$};
\node[font=\tiny] (0010) at (0.4,1) {$e^0_{(j')k'0}$};
\node[font=\tiny] (1010) at (0.4,1.5) {$e^1_{(j')k'0}$};
\node[font=\tiny] (0011) at (0.4,2) {$e^0_{(j')k'1}$};
\node[font=\tiny] (1011) at (0.4,2.5) {$e^1_{(j')k'1}$};
\node[font=\tiny] (0010b) at (0.4,3) {$e^0_{(j')k'0}$};
\node[font=\tiny] (0110) at (3.8,1) {$e^0_{(j'')k''0}$};
\node[font=\tiny] (1110) at (3.8,1.5) {$e^1_{(j'')k''1}$};
\node[font=\tiny] (0111) at (3.8,2) {$e^0_{(j'')k''1}$};
\node[font=\tiny] (1111) at (3.8,2.5) {$e^1_{(j'')k''0}$};
\node[font=\tiny] (0110b) at (3.8,3) {$e^0_{(j'')k''0}$};
\end{tikzpicture}\par\medskip\noindent
\end{center}
\label{fig:Ann}
\caption{Cell decomposition of the annulus neighborhood $\mathfrak{a}_i$ of $\delta_i$}
\end{figure}
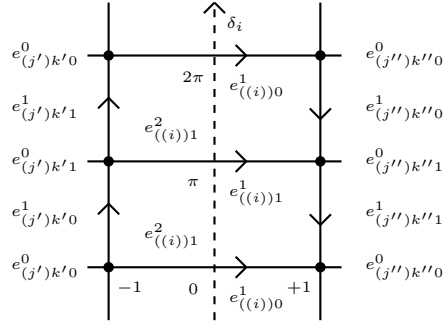

Thus we obtain a cell decomposition associated with the pants decomposition $\cP$. We denote by $\Sigma_{\cP}$ the cell complex $\Sg$ with the cell decomposition. Here it should be remarked that each vertex of $\Sigma_{\cP}$ is located in some pair of pants $P_j$. 
\par
Now we consider a holonomy $\rho$ of a hyperbolic structure on the surface $\Sg$. Its conjugacy class $[\rho]$ is a point of the Teichm\"uller space $\Tg$. 
 It can be regarded as a groupoid cocycle from $\Pi \Sigma\op_{\cP}\vert_{(\Sigma_{\cP})^{(0)}}$ to $\PSL$, where $\Pi \Sigma\op_{\cP}\vert_{(\Sigma_{\cP})^{(0)}}$ is the restriction of the dual category $\Pi \Sigma\op_{\cP}$ of the fundamental groupoid restricted to the vertices $(\Sigma_{\cP})^{(0)}$. The restriction of $\rho$ to each pants $P_j$ can be regarded as a holonomy of the hyperbolic structure on $P_j$ with geodesic boundary. Hence we may consider $\rho\vert_{\Pi {P_j}\op\vert_{(P_j)^{(0)}}}$ to be the standard cocycle by replacing it by its appropriate conjugate, which we denote by $c_\rho$ or $c_{[\rho]}$.  Since each vertex of $\Sigma_{\cP}$ is located in some pair of pants, the cocycle $c_\rho$ is uniquely determined by the point $[\rho] \in \Tg$. 
We call $c_\rho$ the {\it standard cocycle} of the point $[\rho] \in \Tg$. 
For the rest of this paper, we suppose $\rho$ is the standard cocycle of the point $[\rho]$, or equivalently $c_\rho = \rho$.  \par
Now we go back to Figure 3. Let $l_i >0$ be the geodesic length of the simple closed curve $\delta_i$ with respect to the hyperbolic structure $[\rho]$, and 
let $\la_i := \exp(l_i/2)$. Then, from the definition of the standard cocycle on a pair of pants in \S1, we have 
\begin{equation}\label{eq:9.7}
\rho(e^1_{(j')k'0}) = \rho(e^1_{(j')k'1}) = 
\rho(e^1_{(j'')k''0}) = \rho(e^1_{(j'')k''0}) = \bD({\la_i}^{1/2}).
\end{equation}
We denote $A = \rho(e^1_{((i))0})$ and $B = \rho(e^1_{((i))1})$. Then, looking at the $2$-cells $e^2_{((i))0}$ and $e^2_{((i))1}$, we have 
$B\bD({\la_i}^{1/2}) = \bD({\la_i}^{-1/2})A$ and $A\bD({\la_i}^{1/2}) = \bD({\la_i}^{-1/2})B$. If we denote $J := \pm\twomatrix{0}{-1}{1}{0}$, then we have $J^2 = \pm E_2$ and $J\bD({\la_i}^{1/2})J = \bD({\la_i}^{-1/2})$. 
{Here it should be remarked that one cannot distinguish $E_2$ and $-E_2$ as elements of $\PSL$.}
Hence 
$$
J B\bD({\la_i}^{1/2}) = \bD({\la_i}^{1/2})J A, \quad\text{and}\quad
J A\bD({\la_i}^{1/2}) = \bD({\la_i}^{1/2})J B. 
$$
From Lemma \ref{lem:ABBA} (3) we have  
$$
J A = J B = \pm\twomatrix{T_i}{0}{0}{{T_i}\inv},\quad\text{or equivalently}\quad
A = B = \pm \twomatrix{0}{-{T_i}\inv}{T_i}{0}
$$
for some $T_i$. A hyperbolic pair of pants with geodesic boundary has $3$ seams, the shortest geodesics among the geodesic boundary components.
Since the pair admits an orientation-reversing isometry, the foots of seams on a single boundary component divide the component into two parts of the same length. Realize $\delta_i$'s as geodesics with respect to the hyperbolic structure $[\rho]$. Then each pair of pants $P_j$ is a hyperbolic pair of pants with geodesic boundary, and has $3$ seams. In $P$ the $0$-cell $e^0_{k0}$ on $\pa_kP$ is the foot of the seam $e^1_k$ from $\pa_{k-1}P$. 
The twisting parameter $\tau_i\bmod l_i\bZ$ along $\delta_i$ is defined by the oriented distance from the foot of the seam from $\pa_{k'-1}P_{j'}$ to that from $\pa_{k''-1}P_{j''}$ {as in Figure 4}.
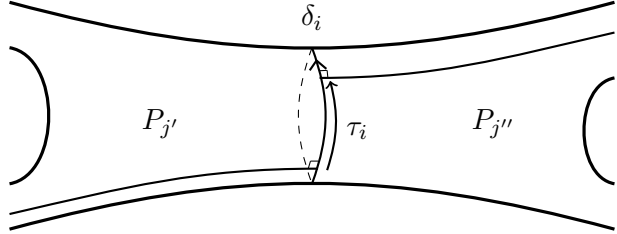
\begin{figure}[ht]
\begin{center}
\begin{tikzpicture}
\tikzset{Process/.style={rectangle,  draw,  text centered, text width=37mm, minimum height=1.5cm}};
\coordinate (UL) at (-4,3);
\coordinate (UM) at (0,2.4);
\coordinate (UR) at (4,3);
\coordinate (DL) at (-4,0);
\coordinate (DM) at (0,0.6);
\coordinate (DR) at (4,0);
\coordinate (LL) at (-4,0.2);
\coordinate (LM) at (0.08,0.8);
\coordinate (RR) at (4,2.6);
\coordinate (RM) at (0.1,2.0);
\coordinate (al) at (-0.03,2.10);
\coordinate (am) at (0.06,2.23);
\coordinate (ar) at (0.19,2.12);
\coordinate (L1) at (-4,2.4); 
\coordinate (L2) at (-4,0.6); 
\coordinate (R1) at (4,2.0); 
\coordinate (R2) at (4,0.6); 
\coordinate (LMr) at (0.20,0.78);
\coordinate (RMr) at (0.24,1.96);
\draw[very thick, bend right = 15] (UL) to (UR);
\draw[very thick, bend left = 15] (DL) to (DR);
\draw[thick, bend left = 20] (UM) to (DM);
\draw[dashed, bend right = 20] (UM) to (DM);
\draw[thick] (LL) to [out=13, in=180] (LM);
\draw[thick] (RR) to [out=193, in=0] (RM);
\draw[thick] (al) to (am);
\draw[thick] (ar) to (am);
\draw[very thick, bend left = 80] (L1) to (L2);
\draw[very thick, bend right = 80] (R1) to (R2);
\draw[->, thick, bend right = 16] (LMr) to (RMr);
\node (D) at (0,2.8) {$\delta_i$};
\node (P1) at (-2,1.4) {$P_{j'}$};
\node (P2) at (2.4,1.4) {$P_{j''}$};
\node (T) at (0.6,1.3) {$\tau_i$};
\draw (-0.02,0.9) to (0.08, 0.9);
\draw (-0.02,0.9) to (-0.05, 0.8);
\draw (0.2,2.1) to (0.21,2.0);
\draw (0.2,2.1) to (0.1,2.095);
\end{tikzpicture}\par\medskip\noindent
\end{center}
\caption{{The oriented distance $\tau_i$.}}
\end{figure}

\begin{thm}\label{thm:twistparam} 
$$
\tau_i \equiv -2\log T_i \pmod{l_i\bZ}.
$$
\end{thm}
\begin{proof}
Take our basepoint at $e^0_{(j')k'0}$, which is located in the left-hand side of $\delta_i$, and denote $\twomatrix{a'}{b'}{c'}{d'} = \rho(e^1_{(j')k'})$ and
$\twomatrix{a''}{b''}{c''}{d''} = \rho(e^1_{(j'')k''})$. By the definition of the standard cocycle, we have $a'b'= c'd'$ and $a''b''= c''d''$. The axis of the holonomy $\rho$ along $\pa_{k'}P_{j'}$ equals the imaginary axis, since the basepoint is located on the boundary. Hence, from $a'b'= c'd'$, $\sqrt{-1}$ is the foot of a seam of $P_{j'}$. \par
Next we compute the foot of the seam from the boundary component $\pa_{k'-1}P_{j'}$, i.e., the foot of the shortest geodesic from 
the axis of 
$\rho(e^1_{((i))0}e^1_{(j'')k''}e^1_{(j'')k''-1,1}e^1_{(j'')k''-1,0}
{e^1_{(j'')k''}}\inv{e^1_{((i))0}}\inv)$. Since 
$$
\rho(e^1_{((i))0}e^1_{(j'')k''}) = \pm\twomatrix{0}{-{T_i}\inv}{T_i}{0}\twomatrix{a''}{b''}{c''}{d''} = \pm \twomatrix{-{T_i}\inv c''}{-{T_i}\inv d''}{{T_i}a''}{{T_i}b''},
$$
the foot is ${T_i}^{-2}\sqrt{-1}$. Hence the oriented distance equals 
$-2\log(T_i)$. This proves the theorem. 
\end{proof}

In the textbook \cite[\S3.2.4]{IT}, the parameter $\tau_i/l_i: \Tg \to \bR$ is defined as a lift of $1/l_i$ times the oriented distance $\tau_i/l_i: \Tg \to \bR/\bZ$ by using the fact $\pi_1(\Tg)$. 
Based on Theorem \ref{thm:twistparam}, we can re-define the twisting parameter 
$$
\tau_i = \tau^{\cP}_i: T_g \mapsto \bR, \quad [\rho] \mapsto -2\log T_i,
$$
for any $1 \leq i \leq 3g-3$. Here $T_i$ is given by $\rho(e^1_{((i))0}) = \pm\twomatrix{0}{-{T_i}\inv}{T_i}{0}$ with respect to the standard cocycle $\rho$ for the point $[\rho] \in \Tg$. 
In particular, our re-definition does not use the fact $\pi_1(\Tg) = 1$. 
Thus we obtain the Fenchel-Nielsen coordinates
$$
\mathrm{FN}^{\cP}:= (l^\cP_1, \dots, l^\cP_{3g-3}, \tau_1^{\cP}, \dots, 
\tau_{3g-3}^{\cP}): T_g \to \bR_{>0}^{3g-3}\times \bR^{3g-3}
$$
with respect to the pants decomposition $\cP$.\par

\section{The first variation of the standard cocycle.}

Now we compute the first variation of the standard cocycle, which will be used in our proof of Wolpert's formula \eqref{eq:Wformula} 
in the next section. The Teichm\"uller space $\Tg$ is an open subset of the moduli $H^1(\Sg; \PSL)$ of flat $\PSL$-bundles on the surface $\Sg$. Hence the tangent space $T_{[\rho]}\Tg$ 
at the point $[\rho] \in \Tg$ is the first cohomology group
\begin{equation}\label{eq:tang}
T_{[\rho]}T_g = H^1(\Sg; \sltwo_{\mathrm{Ad}\rho}).
\end{equation}
Here $\sltwo_{\mathrm{Ad}\rho}$ means the locally constant sheaf on $\Sg$ given by the composite of $\rho$ and the adjoint action $\Ad: \PSL \to GL(\sltwo)$ of the Lie group $\PSL$ on its own Lie algebra $\sltwo$. {See, for example, \cite{G1} and its references.}
\par
Let $\rho_t$, $|t|\ll 1$ with $\rho_0 = \rho$ be a $1$-parameter deformation of the standard cocycle $\rho$ for a point  $[\rho] \in T_g$. 
Under the identification \eqref{eq:tang}, the tangent vector $\overset\cdot\rho := \dfrac{d}{dt}\Bigr\vert_{t=0}[\rho_t] \in T_{[\rho]}T_g$ equals the cohomology class of the cocycle $z(\overset\cdot\rho) \in C^1(\Sigma_\cP; \sltwo_{\mathrm{Ad}\rho})$ on the cell complex $\Sigma_\cP$ defined by 
$$
z(e)(\overset\cdot\rho) := \overset\cdot\rho(e)\rho(e)\inv 
\in (\sltwo_{\mathrm{Ad}\rho})_{\varphi(0)}
$$
for each $1$-cell $e$ whose characteristic map is $\varphi: [0,1] \to 
\overline{e} \subset \Sg$, {where $\overline{e}$ is the closure of the cell $e$}.
This definition follows from our convention of the concatenation. 
{In fact, if $e$ and $e'$ are composable, then Leibniz' rule implies $\overset\cdot\rho(ee')\rho(ee')\inv = 
\overset\cdot\rho(e)\rho(e')\rho(ee')\inv + \rho(e)\overset\cdot\rho(e')\rho(ee')\inv = \overset\cdot\rho(e)\rho(e)\inv + \rho(e)\overset\cdot\rho(e')\rho(e')\inv\rho(e)\inv$, in other words,  $z(e)(\overset\cdot\rho)$ is a $1$-cocycle with values in the left module $\sltwo_{\mathrm{Ad}\rho}$}.
Also we may regard it as an $(\sltwo_{\mathrm{Ad}\rho})_{\varphi(0)}$-vallued $1$-form on $\Tg$
$$
z(e) \in (\sltwo_{\mathrm{Ad}\rho})_{\varphi(0)}\otimes \Omega^1(\Tg),
$$
when $\overset\cdot\rho$ runs over the whole tangent space $T_{[\rho]}\Tg$. 
From the cocycle condition, the $1$ cell $e\inv$ with the reverse orientation satisfies
\begin{equation}\label{eq:invz}
z(e\inv)(\overset\cdot{\rho}) = -\Ad(\rho(e))\inv z(e)(\overset\cdot{\rho}). 
\end{equation}
\par
First, we compute the values $z(e)$ for $1$-cells $e$ inside the pair of pants $P_j$. 
Now let $\la(t) > 0$ be a real valued differentiable function defined on  $|t|\ll 1$. Then we have 
\begin{equation}\label{eq:lat}
\aligned
& \bD(\la(t))\inv \bD(\la(t))' = \bD(\la(t))' \bD(\la(t))\inv
= \twomatrix{\la'(t)}{0}{0}{-\la(t)^{-2}\la'(t)}\twomatrix{\la(t)\inv}{0}{0}{\la(t)}\\
&= \twomatrix{\la'(t)\la(t)\inv}{0}{0}{-\la'(t)\la(t)\inv} 
= (\log\la(t))'\twomatrix{1}{0}{0}{-1}.
\endaligned
\end{equation}
{Here, for any $k \in \bZ/3$, we have $\la(t) = \la_k(t)^{1/2}$ and $l_k = 2 \log \la_k$ for the hyperbolic pants $P$.}
Hence we have 
\begin{equation}\label{eq:1lk}
z(e^1_{(j)k0}) = z(e^1_{(j)k1}) = \frac14\twomatrix{1}{0}{0}{-1} dl_{(j)k},
\end{equation}
where $l_{(j)k}$ is the geodesic length function of 
the boundary component $\pa_kP_j$. By \eqref{eq:invz} we have 
\begin{equation}\label{eq:dliinv}
\Ad(\rho(e^1_{(j)k\epsilon}))z(e^1_{(k\epsilon}) = 
 -z((e^1_{(j)k\epsilon})\inv) = \frac14\twomatrix{1}{0}{0}{-1} dl_k
\end{equation}
for $\epsilon = 0,1$. 
\par
As was proved in Theorem \ref{thm:A_k} {for $e^1_k = e^1_{(j)k}$}, we have
$$
\rho(e^1_{(j)k}) = A_k = \pm\twomatrix{\sqrt{|b_kc_k|-1}}{|b_kc_k|^{1/2}}{{-}|b_kc_k|^{1/2}}{-\sqrt{|b_kc_k|-1}},
$$
where $|b_kc_k|$ is given by substituting $l_{\hat k} = l_{(j)\hat k}$ into \eqref{eq:bkck}. 
\begin{lem}\label{lem:seamBz}
$$
\aligned 
& z(e^1_{(j)k}) = \frac12\sqrt{|b_kc_k|(|b_kc_k|-1)\inv}\twomatrix{0}{1}{1}{0}d\log|b_kc_k|,\\
& \Ad(\rho(e^1_{(j)k}))z(e^1_{(j)k}) 
= -\frac12\sqrt{|b_kc_k|(|b_kc_k|-1)\inv}\twomatrix{0}{1}{1}{0}d\log|b_kc_k|.
\endaligned
$$
\end{lem}
\begin{proof} Let $f(t) > 1$ be a real-valued differentiable function defined on  $|t|\ll 1$, and consider $B(t) =\twomatrix{(f(t)-1)^{1/2}}{f(t)^{1/2}}{{-}f(t)^{1/2}}{-(f(t)-1)^{1/2}}$. Then we have $B(t)^2 = -I$, and 
$$
\aligned
&B'(t)B(t)\inv = - B'(t)B(t)\\
& = {+}\frac12f'(t)\twomatrix{(f(t)-1)^{-1/2}}{{+}f(t)^{-1/2}}{{-}f(t)^{-1/2}}{-(f(t)-1)^{-1/2}}\twomatrix{(f(t)-1)^{1/2}}{f(t)^{1/2}}{{-}f(t)^{1/2}}{-(f(t)-1)^{1/2}}\\
&= -\frac12f'(t)(f(t)^{-1/2}(f(t)-1)^{1/2} - f(t)^{1/2}(f(t)-1)^{-1/2})
\twomatrix{0}{1}{1}{0}\\
&= -\frac12f'(t)f(t)^{-1/2}(f(t)-1)^{-1/2}(f(t)-1 - f(t))
\twomatrix{0}{1}{1}{0}\\
&= \frac12f'(t)f(t)^{-1/2}(f(t)-1)^{-1/2}
\twomatrix{0}{1}{1}{0}\\
&= \frac12f(t)^{1/2}(f(t)-1)^{-1/2}
\twomatrix{0}{1}{1}{0} (\log(f(t)))'.
\endaligned
$$
On the other hand, $\Ad(B(t))(B'(t)B(t)\inv) = B(t)B'(t)B(t)\inv B(t)\inv
= {-}B(t)B'(t) = B'(t)B(t) = {-}B'(t)B(t)\inv$. This proves the lemma. 
\end{proof}

\par
Next we compute $z(e^1_{((i))\epsilon})$, $\epsilon = 0,1$, 
which crosses the simple closed curve $\delta_i$, $1 \leq i \leq 3g-3$. As was shown in \S2, we have $\rho(e^1_{((i))\epsilon}) = \bD(\exp(\tau_i/2))J$, where $J = 
\pm\twomatrix{0}{-1}{1}{0}$. Hence, {if we denote by 
$\bD(\exp(\tau_i/2))^\cdot$ the first variation of $\bD(\exp(\tau_i/2))$ 
along the tangent vector $\overset\cdot{\rho}$}, we obtain
$$
z(e^1_{((i))\epsilon})(\overset\cdot{\rho}) = \bD(\exp(\tau_i/2))^\cdot JJ\inv \bD(\exp(-\tau_i/2))
= \frac12\twomatrix{1}{0}{0}{-1}(d\tau_i)(\overset\cdot{\rho}),
$$
i.e., 
\begin{equation}\label{eq:dtaui}
z(e^1_{((i))\epsilon}) = \frac12\twomatrix{1}{0}{0}{-1}d\tau_i.
\end{equation}

\section{The Weil-Petersson symplectic form}

In this section we prove Wolpert's formula \eqref{eq:Wformula} on the Weil-Petersson symplectic form $\WP$ by using the standard cocycle introduced in the previous sections. By Goldman \cite{G1, G2} {(see also Proposition \ref{prop:WPABG})}, the symplectic form $\WP$ equals {twice} the intersection form on the surface $\Sg$ involved with the invariant non-degenerate symmetric form $\Btr(X, Y) = \trace(XY)$, 
$X, Y \in \sltwo$, on the Lie algebra $\sltwo$. 
If $X = \twomatrix{a}{b}{c}{-a}$  and $Y = \twomatrix{p}{q}{r}{-p}$, then we have 
$$
\Btr\left(\twomatrix{a}{b}{c}{-a}, \twomatrix{p}{q}{r}{-p}\right) = 
\trace\left(\twomatrix{a}{b}{c}{-a}\twomatrix{p}{q}{r}{-p}\right) = 2ap + br+cq.
$$
As a corollary of Lemma \ref{lem:seamBz}, we have 
\begin{cor}\label{cor:orthoB} $z(e^1_{(j)k})$ and $\Ad(\rho(e^1_{(j)k}))z(e^1_{(j)k})$ are orthogonal to 
$\twomatrix{1}{0}{0}{-1}$ with respect to the symmetric form $\Btr$. 
\end{cor}
\par

{Hence} what we like to compute is 
\begin{equation}\label{eq:WP}
\frac12\WP(\overset\cdot{\rho'}, \overset\cdot{\rho''})
= \left\langle {\Btr}_*\left(z(-)(\overset\cdot{\rho'})\times  z(-)(\overset\cdot{\rho''})\right), \widetilde{\Delta}([\Sg])\right\rangle
\end{equation}
for any tangent vectors $\overset\cdot{\rho'}, \overset\cdot{\rho''} \in T_{[\rho]} T_g$. The left-hand side of the pairing is a $2$-cocycle with trivial coefficients, 
and the right-hand side $\widetilde{\Delta}([\Sg])$ the $2$-cycle with trivial coefficients induced by a cellular approximation $\widetilde{\Delta}$ of the diagonal map $\Delta: \Sg \to \Sg\times \Sg$. {(See, e.g., \cite{May} Chapter 10, \S4, pp.74-75 and the last paragraph of Chapter 18, \S3, p.138.)}
We will give an explicit cellular approximation $\widetilde{\Delta}$ on each of the squares and the hexagons in the decomposition $\Sigma_{\cP}$, and compute the pairing using the approximations. As will be shown, the contribution of the annuls neighborhood of $\delta_i$ equals $d\tau_i \wedge d l_i$. \par
In the cell complex $\Sigma_{\cP}$, the closure $\overline{E}$ of each $2$-cell (= face) $E$ is a closed disk whose boundary is a $1$-dimensional subcomplex. 
Suppose the cross product $e'\times e''$ of some $1$-cells $(e', \varphi')$ and $(e'', \varphi'')$ in the boundary of $\overline{E}$ appears in the $2$-chain $\widetilde{\Delta}(E)$. Take a $0$-cell $*$ in the closure $\overline{E}$, and 
paths $\ell'$ (resp.\ $\ell''$) running from $\varphi'(0)$ (resp.\ $\varphi''(0)$) to $*$ inside $\overline{E}$. Then, for any $1$-cocycles $z', z'' \in Z^1(X; \sltwo_{\Ad\rho})$, we have 
\begin{equation}\label{eq:evalB}
\langle {\Btr}_*(z'\times z''), e'\times e''\rangle = \Btr\bigl((\Ad\rho)(\ell')(z'(e')), (\Ad\rho)(\ell'')(z'(e''))\bigr),
\end{equation}
which is independent of the choice of $*$, $\ell'$ and $\ell''$, since 
$\overline{E}$ is contractible and $\Btr$ is $\PSL$-invariant. 
\par

Next we explicitly construct a celluar approximation $\widetilde{\Delta}$ of the diagonal map $\Delta: X \to X\times X$ with respect to the cell decomposition $X = \Sigma_\cP = \Sg$.
In the cell decomposition, the characteristic map for each cell is 
a homeomorphism onto a subcomplex of $X$. 
Hence any cellular map $\widetilde{\Delta}: X \to X\times X$, 
which sarisfies the condition 
\begin{equation}\label{eq:acyclic}
\widetilde{\Delta}(\overline{e}) \subset \overline{e}\times\overline{e} \,\,(\subset X\times X)
\end{equation}
for each cell $e$, is homotopic to the diagonal map $\Delta$.\par
For any $0$-cell $x_0 \in X$ we have to define 
$$
\widetilde{\Delta}(x_0) = (x_0, x_0)
$$
by the condition \eqref{eq:acyclic}. 
For any $1$-cell $e$ whose characteristic map is $\varphi: [0, 1] \to X$, we define 
\begin{equation}\label{eq:1cell}
\widetilde{\Delta}(\varphi(t)) := \begin{cases}
(\varphi(2t), \varphi(0)), & \text{if $0 \leq t \leq 1/2$},\\
(\varphi(1), \varphi(2t-1)), & \text{if $1/2 \leq t \leq 1$}.\\
\end{cases}
\end{equation}
Here we remark that this definition depends on the orientation of the cell $e$. In order to make our calculation for the squares simpler, we choose $(e^1_{(j)k1})\inv$ instead of $e^1_{(j)k1}$ for any $1 \leq j \leq 3g-3$ and $k \in \bZ/3$. 
\par
We have two kinds of $2$-cells $E$: squares and hexagons. 
Here it should be remarked the approximation was already determined  on the boundary $\pa E$ based on the orientation of each $1$-cell appearing there. \par
For each square $E =: E_S$, we take an approximation $\widetilde{\Delta}$ as in Figure {5}. Here $- g_1\times f_1$ means that the orientation of $g_1\times f_1$ is converse to that of the square $E_S$. 
\begin{figure}
\begin{center}
\begin{tikzpicture}
\tikzset{Process/.style={rectangle,  draw,  text centered, text width=37mm, minimum height=1.5cm}};
\coordinate (00) at (0,1) node at (00) [below left] {$p_{00}$};
\coordinate (10) at (2,1) node at (10) [below right] {$p_{10}$};
\coordinate (01) at (0,3) node at (01) [above left] {$p_{01}$};
\coordinate (11) at (2,3) node at (11) [above right] {$p_{11}$};
\fill (00) circle (2pt) (10) circle (2pt) (01) circle (2pt) (11) circle (2pt) ;
\draw[thick] (00) -- (10) node [midway] {$>$};
\draw[thick] (01) -- (11) node [midway] {$>$};
\draw[thick] (00) -- (01) node [midway] {$\wedge$};
\draw[thick] (10) -- (11) node [midway] {$\wedge$};
\node (*) at (1,2) {$E_S$};
\node () at (1,0.7) {$f_0$};
\node () at (1,3.3) {$f_1$};
\node () at (-0.3,2) {$g_0$};
\node () at (2.3,2) {$g_1$};
\coordinate (0000) at (6,0) node at (0000) [below left] {\tiny $(p_{00}, p_{00})$};
\coordinate (0100) at (6,2) node at (0100) [left] {\tiny $(p_{01}, p_{00})$};
\coordinate (0101) at (6,4) node at (0101) [above left] {\tiny $(p_{01}, p_{01})$};;
\coordinate (1000) at (8,0) node at (1000) [below] {\tiny $(p_{10}, p_{00})$};
\coordinate (1100) at (8,2) node at (1100) [above right] {\tiny $(p_{11}, p_{00})$};
\coordinate (1101) at (8,4) node at (1101) [above] {\tiny $(p_{11}, p_{01})$};
\coordinate (1010) at (10,0) node at (1010) [below right] {\tiny $(p_{10}, p_{10})$};
\coordinate (1110) at (10,2) node at (1110) [right] {\tiny $(p_{11}, p_{10})$};
\coordinate (1111) at (10,4) node at (1111) [above right] {\tiny $(p_{11}, p_{11})$};;
\fill (0000) circle (2pt) (0100) circle (2pt) (0101) circle (2pt) (1000) circle (2pt) (1100) circle (2pt) (1101) circle (2pt) (1010) circle (2pt) (1110) circle (2pt) (1111) circle (2pt);
\draw[thick] (0000) -- (0100);
\draw[thick] (0100) -- (0101);
\draw[thick] (0000) -- (1000);
\draw[thick] (0100) -- (1100);
\draw[thick] (0101) -- (1101);
\draw[thick] (1000) -- (1100);
\draw[thick] (1100) -- (1101);
\draw[thick] (1000) -- (1010);
\draw[thick] (1100) -- (1110);
\draw[thick] (1101) -- (1111);
\draw[thick] (1010) -- (1110);
\draw[thick] (1110) -- (1111);
\node () at (7,1) {$E_S\times p_{0,0}$};
\node () at (9,1) {$-g_1\times f_0$};
\node () at (7,3) {$f_1\times g_0$};
\node () at (9,3) {$p_{1,1} \times E_S$};
\end{tikzpicture}\par\medskip\noindent
\end{center}
\label{fig:Ann}
\caption{Celluar approximation of the diagonal map on the square $E_S$}
\end{figure}
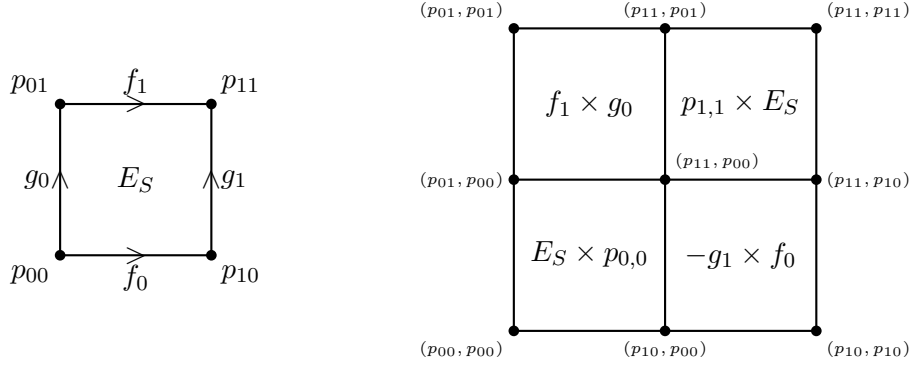
This means
$$
\widetilde{\Delta}(E_S) = f_1\times g_0 + p_{1,1}\times E_S + E_S\times p_{0,0}
{- g_1\times f_0} \in C_*(X\times X) = C_*(X)\otimes C_*(X).
$$

\begin{lem}\label{lem:icont}
The contribution of the $2$-cells $e^2_{((i))0}$ and $e^2_{((i))1}$ 
to the Weil-Petersson symplectic form $\WP$ is
$d\tau_i\wedge dl_i$.
\end{lem}
\begin{proof}
The terms $p_{1,1}\times E_S$ and $E_S\times p_{0,0}$ have no contribution to the cup product. In order to compute the other terms, 
we use \eqref{eq:1lk} and \eqref{eq:dtaui}. Recall 
$\Btr\left(\twomatrix{1}{0}{0}{-1}\twomatrix{1}{0}{0}{-1}\right) = 2$. \par
For $E_S = e^2_{((i))0}$, we take $f_0 = e^1_{((i))0}$, $f_1 = e^1_{((i))1}$, 
$g_0 = e^1_{(j')k'0}$ and $g_1 = (e^1_{(j'')k''1})\inv$. 
When we compute $-g_1\times f_0$, from \eqref{eq:evalB}, we have to move $z(f_0)$ along $f_0$, which changes the sign. Hence we obtain $-\frac14dl_i(\overset\cdot{\rho'})d\tau_i(\overset\cdot{\rho''})$.
On the other hand, for $f_1\times g_0$, we move $z(g_0)$ along 
$g_0$, which does not change the sign. Hence we obtain 
$+\frac14d\tau_i(\overset\cdot{\rho'})dl_i(\overset\cdot{\rho''})$.\par
For $E_S = e^2_{((i))1}$, we take $f_0 = (e^1_{(j')k'1})\inv$, $f_1 = e^1_{(j'')k''0}$, $g_0 = e^1_{((i))0}$ and $g_1 = e^1_{((i))1}$.
Then we get $- \frac14dl_i(\overset\cdot{\rho'})d\tau_i(\overset\cdot{\rho''})
+ \frac14d\tau_i(\overset\cdot{\rho'})dl_i(\overset\cdot{\rho''})$.
Here the sign for $f_1\times g_0$ is changed since we need to move $z(g_0)$ along $g_0$. \par
Hence the contribution of the $2$-cells $e^2_{((i))0}$ and $e^2_{((i))1}$
{to $\frac12\WP$}  is
\begin{equation}\label{eq:icont}
\frac12(d\tau_i(\overset\cdot{\rho'})dl_i(\overset\cdot{\rho''})
- dl_i(\overset\cdot{\rho'})d\tau_i(\overset\cdot{\rho''}))
= {\frac12}(d\tau_i\wedge dl_i)(\overset\cdot{\rho'}, \overset\cdot{\rho''}).
\end{equation}
This proves the lemma.
\end{proof}

Now we consider a hexagon $E =: E_H$ whose vertices are the sixth roots of unity $\{1, -\omega^2, \omega, -1, \omega^2, -\omega\}$ in the complex line $\bC$. 
The orientation of each $1$-cell goes from the left to the right, and 
we define the same cellular approximation $\widetilde{\Delta}$
on the $0$- and $1$-cells as above. Then we can extend the map $\widetilde{\Delta}$ to the interior of $E_H$ as in Figure {6}. {In particular, the extended map on each rhombus identified with $[0,1]\times [0,1]$ is the product map of the characteristic maps of the corresponding $1$-cells. For example, $f_3\times f_1$ is a composite map of the product map $[0,1]\times [0,1] \to X\times X$, $(t_1, t_2) \mapsto (f_3(t_1), f_1(t_2))$, and a homeomorphism from the rhombus onto $[0,1]\times [0,1]$ which maps 
$(-\omega^2,-1)$, $(1,-1)$, $(-\omega^2,\omega)$ and $(1, \omega)$ to 
$(0,0)$, $(1,0)$, $(0,1)$ and $(1,1)$, respectively.  
The sign tells us if the product map is orientation-preserving or not.}
\begin{figure}[ht]
\begin{center}
\begin{tikzpicture}
\tikzset{Process/.style={rectangle,  draw,  text centered, text width=37mm, minimum height=1.5cm}};
\coordinate (+o2) at (1,1.732) 
node at (+o2) [below left] {$\omega^2$};
\coordinate (-o) at (3,01.732)
node at (-o) [below right] {$-\omega$};
\coordinate (-1) at (0,3.464)
node at (-1) [left] {$-1$};
\coordinate (+1) at (4,3.464)
node at (+1) [right] {$+1$};
\coordinate (+o) at (1,5.206)
node at (+o) [above left] {$\omega$};
\coordinate (-o2) at (3,5.206)
node at (-o2) [above right] {$-\omega^2$};
\fill (+o2) circle (2pt) (-o) circle (2pt)
(-1) circle (2pt) (+1) circle (2pt)
(+o) circle (2pt) (-o2) circle (2pt);
\draw[thick] (-1) -- (+o) node [midway] [sloped] {$>$}
node [midway] [above left] {$f_1$};
\draw[thick] (-1) -- (+o2) node [midway] [sloped] {$>$}
node [midway] [below left] {$g_1$};
\draw[thick] (+o) -- (-o2) node [midway] [sloped] {$>$}
node [midway] [above] {$f_2$};
\draw[thick] (+o2) -- (-o) node [midway] [sloped] {$>$}
node [midway] [below] {$g_2$};
\draw[thick] (-o2) -- (+1) node [midway] [sloped] {$>$}
node [midway] [above right] {$f_3$};
\draw[thick] (-o) -- (+1) node [midway] [sloped] {$>$}
node [midway] [below right] {$g_3$};
\node () at (2,3.464) {$E_H$};
\coordinate (+o2+o2) at (8,0)
node at (+o2+o2) [below left] {\tiny $(\omega^2,\omega^2)$};
\coordinate (-o+o2) at (10,0)
node at (-o+o2) [below] {\tiny $(-\omega,\omega^2)$};
\coordinate (-o-o) at (12,0)
node at (-o-o) [below right] {\tiny $(-\omega,-\omega)$};
\coordinate (+o2-1) at (7,1.732)
node at (+o2-1) [left] {\tiny $(\omega^2,-1)$};
\coordinate (-o-1) at (9,1.732)
node at (-o-1) [above left] {\tiny $(-\omega,-1)$};
\coordinate (+1+o2) at (11,1.732)
node at (+1+o2) [above right] {\tiny $(1,\omega^2)$};
\coordinate (+1-o) at (13,1.732)
node at (+1-o) [right] {\tiny $(1,-\omega)$};
\coordinate (-1-1) at (6,3.464)
node at (-1-1) [right] {\tiny $(-1,-1)$};
\coordinate (+1-1) at (10,3.464)
node at (+1-1) [right] {\tiny $(1,-1)$};
\coordinate (+1+1) at (14,3.464)
node at (+1+1) [right] {\tiny $(1,1)$};
\coordinate (+o-1) at (7,5.206)
node at (+o-1) [left] {\tiny $(\omega,-1)$};
\coordinate (-o2-1) at (9,5.206)
node at (-o2-1) [below left] {\tiny $(-\omega^2,-1)$};
\coordinate (+1+o) at (11,5.206)
node at (+1+o) [below right] {\tiny $(1,\omega)$};
\coordinate (+1-o2) at (13,5.206)
node at (+1-o2) [right] {\tiny $(1,-\omega^2)$};
\coordinate (+o+o) at (8,6.928)
node at (+o+o) [above left] {\tiny $(\omega,\omega)$};
\coordinate (-o2+o) at (10,6.928)
node at (-o2+o) [above] {\tiny $(-\omega^2,\omega)$};
\coordinate (-o2-o2) at (12,6.928)
node at (-o2-o2) [above right] {\tiny $(-\omega^2,-\omega^2)$};
\fill (+o2+o2) circle (2pt) (-o+o2) circle (2pt) (-o-o) circle (2pt)
(+o2-1) circle (2pt) (-o-1) circle (2pt) 
(+1+o2) circle (2pt) (+1-o) circle (2pt) 
(-1-1) circle (2pt) (+1-1) circle (2pt) (+1+1) circle (2pt)
(+o-1) circle (2pt) (-o2-1) circle (2pt) 
(+1+o) circle (2pt) (+1-o2) circle (2pt) 
(+o+o) circle (2pt) (-o2+o) circle (2pt) (-o2-o2) circle (2pt);
\draw[thick] (+o2+o2) -- (-o-o) 
(+o2-1) -- (+o2+o2) (-o-1) -- (-o+o2) (+1+o2) -- (-o+o2) (+1-o) -- (-o-o)
(+o2-1) -- (-o-1) (+1+o2) -- (+1-o)
(-1-1) -- (+o2-1) (-o-1) -- (+1-1) (+1-1) -- (+1+o2) (+1-o) -- (+1+1)
(-1-1) -- (+o-1) (-o2-1) -- (+1-1) (+1-1) -- (+1+o) (+1-o2) -- (+1+1)
(+o-1) -- (-o2-1) (+1+o) -- (+1-o2)
(+o-1) -- (+o+o) (-o2-1) -- (-o2+o) (+1+o) -- (-o2+o) (+1-o2) -- (-o2-o2)
(+o+o) -- (-o2-o2) ;
\node () at (8.5, 0.866) {$-g_2\times g_1$};
\node () at (11.5, 0.866) {$-g_3\times g_2$};
\node () at (10.0, 1.732) {$-g_3\times g_1$};
\node () at (8.2, 3.464) {$E_H\times\{-1\}$};
\node () at (12.0, 3.464) {$\{1\}\times E_H$};
\node () at (10.0, 5.206) {$f_3\times f_1$};
\node () at (8.5, 6.072) {$f_2\times f_1$};
\node () at (11.5,6.072) {$f_3\times f_2$};
\end{tikzpicture}
\end{center}
\caption{{C}ellular approximation of the diagonal map on the hexagon $E_H$}
\end{figure}
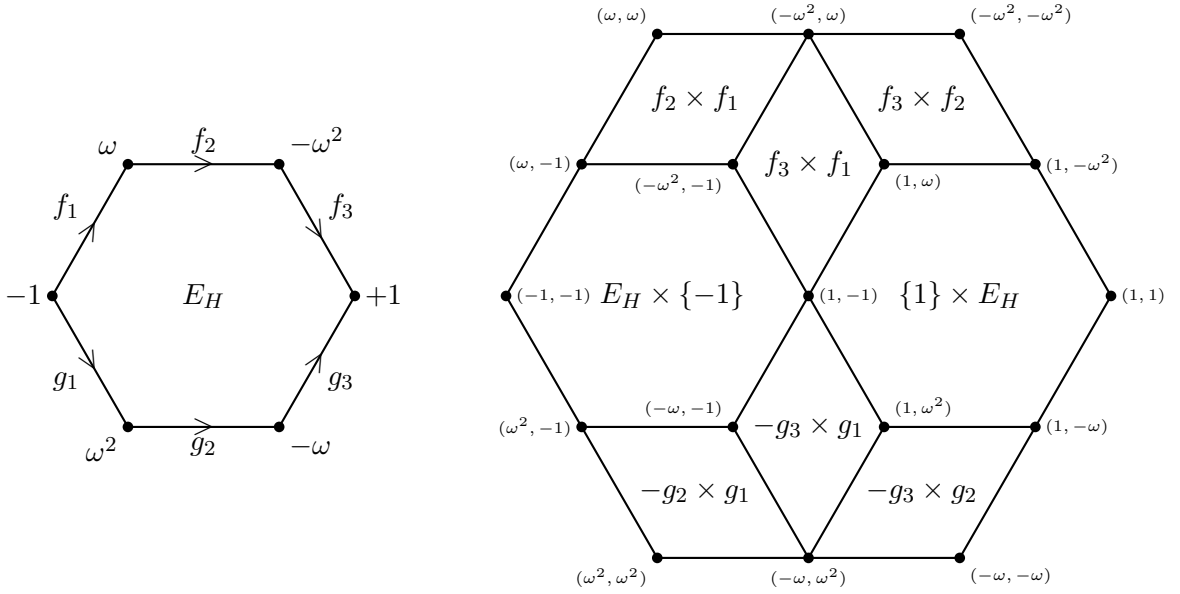
In other words, we have 
\begin{equation}\label{eq:hexagon}
\widetilde{\Delta}(E_H) = 
\{1\}\times E_H + E_H\times \{-1\} 
+ f_3\times f_2 + f_3\times f_1 + f_2\times f_1
- g_3\times g_2 - g_3\times g_1 - g_2\times g_1
\end{equation}
in the celluar chain complex $C_*(\overline{E_H}\times \overline{E_H}) = C_*(\overline{E_H})\otimes C_*(\overline{E_H})$. 
\par
We compute the contribution of each pair of pants $P = P_j$, $1 \leq j \leq 2g-2$, to the Weil-Petersson symplectic form $\WP$ in \eqref{eq:WP}. In what follows, we drop the subscript $j$.
In order to make the $2$-cells $e^2_{\pm}$ fit the hexagon $E_H$, 
we reverse the orientations of $e^1_0$ and $e^1_{10}$, and 
put $e^1_{10}$ back to the original orientation. 
In order to compute the effect coming from the change of the orientations of $e^1_{10}$ and $e^1_{11}$, we insert two bigons 
along the boundary component $\pa_1P$ as in Figure {7}.\par
\begin{figure}[ht]
\begin{center}
\begin{tikzpicture}
\tikzset{Process/.style={rectangle,  draw,  text centered, text width=37mm, minimum height=1.5cm}};
\coordinate (010) at (0,2) node at (010) [left] {$e^0_{10}$};
\coordinate (011) at (4,2) node at (011) [right] {$e^0_{11}$};
\fill (010) circle (2pt) (011) circle (2pt);
\coordinate (ddl) at (0,0); 
\coordinate (dl) at (0,1); 
\coordinate (ul) at (0,3); 
\coordinate (uul) at (0,4); 
\coordinate (ddr) at (4,0); 
\coordinate (dr) at (4,1); 
\coordinate (ur) at (4,3); 
\coordinate (uur) at (4,4); 
\draw[thick, rounded corners=20pt] (010) -- (ul) -- (ur) 
node [midway] [sloped] {$>$} node [midway] [below] {$e^1_{10}$} 
-- (011)  ;
\draw[rounded corners=20pt] (010) -- (uul) -- (uur) 
 node [midway] [above] {$(e^1_{10})\inv$} 
-- (011)  ;
\draw[thick, rounded corners=20pt] (010) -- (dl) -- (dr) 
node [midway] [above] {$(e^1_{11})\inv$} -- (011)  ;
\draw[thick, rounded corners=20pt] (010) -- (ddl) -- (ddr) 
node [midway] [below] {$e^1_{11}$}-- (011)  ;
\filldraw[fill = lightgray, thick, rounded corners=20pt] (010) -- (ul) -- (ur) node [midway] [sloped] {$>$} -- (011) -- (uur) -- (uul) node [midway] [sloped] {$<$}-- (010) ;
\filldraw[fill = gray, thick, rounded corners=20pt] (010) -- (dl) -- (dr) node [midway] [sloped] {$>$} -- (011) -- (ddr) -- (ddl) node [midway] [sloped] {$<$}-- (010) ;
\node () at (2,2) {$\pa_1P$};
\coordinate (a0) at (7,3) node at (a0) [above left] {\tiny $(e^0_{10}, e^0_{11})$}; 
\coordinate (b0) at (9,3) node at (b0) [above right] {\tiny $(e^0_{11}, e^0_{11})$}; 
\coordinate (c0) at (7,1) node at (c0) [below left] {\tiny $(e^0_{10}, e^0_{10})$}; 
\coordinate (d0) at (9,1) node at (d0) [below right] {\tiny $(e^0_{11}, e^0_{10})$}; 
\filldraw[fill=lightgray, thick] (a0) -- (b0) node [midway] {$<$} -- (d0) 
node [midway] [sloped] {$<$}
-- (c0) node [midway] {$>$} -- (a0) node [midway] [sloped] {$<$};
\fill (a0) circle (2pt) (b0) circle (2pt) (c0) circle (2pt) (d0) circle (2pt);
\node () at (8,2) {\tiny $e^1_{10}\times e^1_{10}$};
\coordinate (a1) at (12,3) node at (a1) [above left] {\tiny $(e^0_{11}, e^0_{10})$}; 
\coordinate (b1) at (14,3) node at (b1) [above right] {\tiny $(e^0_{11}, e^0_{11})$}; 
\coordinate (c1) at (12,1) node at (c1) [below left] {\tiny $(e^0_{10}, e^0_{10})$}; 
\coordinate (d1) at (14,1) node at (d1) [below right] {\tiny $(e^0_{10}, e^0_{11})$}; 
\filldraw[fill=gray, thick] (a1) -- (b1) node [midway] {$>$} -- (d1) 
node [midway] [sloped] {$>$}
-- (c1) node [midway] {$<$} -- (a1) node [midway] [sloped] {$>$};
\fill (a1) circle (2pt) (b1) circle (2pt) (c1) circle (2pt) (d1) circle (2pt);
\node () at (13,2) {\tiny $-e^1_{11}\times e^1_{11}$};
\end{tikzpicture}
\end{center}
\caption{two bigons inserted along the boundary component $\pa_1P$}
\end{figure}
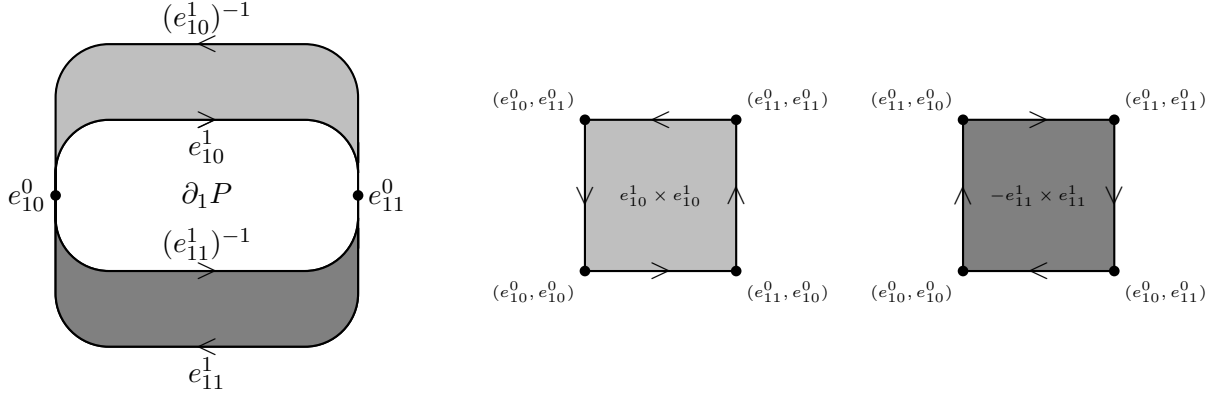
The map $\widetilde{\Delta}$ defined on the boundary of the two bigons as above extends to the interior of the bigons as in Figure {7}. 
Then it maps the lightgray bigon onto $e^1_{10}\times e^1_{10}$ 
in an orientation-preserving way. Hence, by \eqref{eq:1lk}, the contribution of the lightgray bigon is $\dfrac18 dl_1(\overset\cdot{\rho'})dl_1(\overset\cdot{\rho''})$. 
On the other hand,  it maps the gray bigon onto $e^1_{11}\times e^1_{11}$ in an orientation-reversing way. Hence, by \eqref{eq:1lk}, the contribution of the gray bigon is $-\dfrac18 dl_1(\overset\cdot{\rho''})dl_1(\overset\cdot{\rho'})$.
Consequently the contributions cancel each other.\par
\begin{figure}[ht]
\begin{center}
\begin{tikzpicture}[scale=0.75, transform shape]
\tikzset{Process/.style={rectangle,  draw,  text centered, text width=37mm, minimum height=1.5cm}};
\coordinate (A) at (0,6);
\coordinate (B) at (14,6);
\coordinate (C) at (0,0);
\coordinate (D) at (14,0);
\coordinate (p) at (2,4);
\coordinate (q) at (6,4);
\coordinate (r) at (2,2);
\coordinate (s) at (6,2);
\coordinate (x) at (8,4);
\coordinate (y) at (12,4);
\coordinate (z) at (8,2);
\coordinate (w) at (12,2);
\coordinate (001) at (0,3);
\coordinate (010) at (2,3);
\coordinate (011) at (6,3);
\coordinate (020) at (8,3);
\coordinate (021) at (12,3);
\coordinate (000) at (14,3);
\fill (000) circle (2pt) (001) circle (2pt)
(010) circle (2pt) (011) circle (2pt)
(020) circle (2pt) (021) circle (2pt);
\draw[thick] (001) -- (010) node [midway] {$<$}
node [midway] [above] {$e^1_1 = f_{3+}$} 
node [midway] [below] {$e^1_1 = g_{3-}$} ;
\draw[thick] (011) -- (020) node [midway] {$<$}
node [midway] [above] {$e^1_2 = f_{1+}$} 
node [midway] [below] {$e^1_2 = g_{1-}$} ;
\draw[thick] (021) -- (000) node [midway] {$>$}
node [midway] [above] {$(e^1_0)^{-1} = g_{2+}$} 
node [midway] [below] {$(e^1_0)^{-1} = f_{2-}$} ;
\node at (001)  [left] {$e^0_{01}$}; 
\node at (000) [right] {$e^0_{00}$}; 
\node at (011)  [left] {$e^0_{11}$}; 
\node at (010) [right] {$e^0_{10}$}; 
\node at (021)  [left] {$e^0_{21}$}; 
\node at (020) [right] {$e^0_{20}$}; 
\draw[thick, rounded corners=30pt] (001) -- (A) -- (B) node [midway] {$<$} node [midway] [below] {$e^1_{00} = g_{3+}$} -- (000);
\draw[thick, rounded corners=30pt] (001) -- (C) -- (D) node [midway] {$<$} node [midway] [above] {$(e^1_{01})^{-1} = f_{3-}$} -- (000);
\draw[thick, rounded corners=20pt] (010) -- (p) -- (q) node [midway]  {$<$} node [midway] [above] {$(e^1_{10})^{-1} = f_{2+}$} -- (011);
\draw[thick, rounded corners=20pt] (010) -- (r) -- (s) node [midway]  {$<$} node [midway] [below] {$e^1_{11} = g_{2-}$} -- (011);
\draw[thick, rounded corners=20pt] (020) -- (x) -- (y) node [midway] {$>$} node [midway] [above] {$e^1_{20} = g_{1+}$} -- (021);
\draw[thick, rounded corners=20pt] (020) -- (z) -- (w) node [midway] {$>$} node [midway] [below] {$(e^1_{21})^{-1} = f_{1-}$}  -- (021);
\node () at (7,4.5) {$e^2_+$};
\node () at (7,1.5) {$e^2_-$};
\node () at (4,3) {$\partial_1P$};
\node () at (10,3) {$\partial_{{2}}P$};
\node () at (15,2) {$\partial_0P$};
\end{tikzpicture}
\end{center}
\caption{Identification of $e^2_{\pm}$ with the hexagon $E_H$.}
\end{figure}
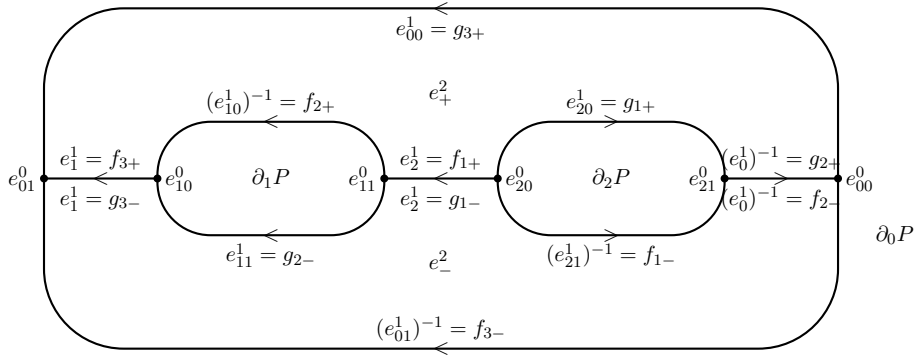

Now we regard the $2$-cells $e^2_+$ and $e^2_-$ as hexagons as in Figure {8}. 
Then we have 
$$
\aligned
&\widetilde{\Delta}([e^2+] + [e^2_-])\\
& = \{e^0_{01}\}\times ([e^2+] + [e^2_-]) 
+ ([e^2+] + [e^2_-])\times \{e^0_{20}\} \\
& \quad 
+ f_{3+}\times f_{2+} + f_{3+}\times f_{1+} + f_{2+}\times f_{1+}
- g_{3+}\times g_{2+} - g_{3+}\times g_{1+} - g_{2+}\times g_{1+}\\
&\quad
+ f_{3-}\times f_{2-} + f_{3-}\times f_{1-} + f_{2-}\times f_{1-}
- g_{3-}\times g_{2-} - g_{3-}\times g_{1-} - g_{2-}\times g_{1-}
\endaligned
$$
The first line of the RHS does not contribute to the Weil-Petersson symplectic form in \eqref{eq:WP}. 
From Corollary \ref{cor:orthoB} and \eqref{eq:evalB}, the contribution of each of 
$f_{3\pm}\times f_{2\pm}$, 
$f_{2\pm}\times f_{1\pm}$, $g_{3\pm}\times g_{2\pm}$ and $g_{2\pm}\times g_{1\pm}$ in \eqref{eq:WP} vanishes. 
Moreover, we have
$$
\aligned
& \left\langle {\Btr}_*\left(z(-)(\overset\cdot{\rho'})\times  z(-)(\overset\cdot{\rho''})\right), g_{3+}\times g_{1+}\right\rangle
= \frac14\Btr\left(\Ad(\rho(e^1_0))\twomatrix{1}{0}{0}{-1}, 
\twomatrix{1}{0}{0}{-1}\right)dl_0(\overset\cdot{\rho'}) dl_2(\overset\cdot{\rho''})
\\
& = \left\langle {\Btr}_*\left(z(-)(\overset\cdot{\rho'})\times  z(-)(\overset\cdot{\rho''})\right), f_{3-}\times f_{1-}\right\rangle.
\endaligned
$$
Hence there remains only $f_{3+}\times f_{1+} - g_{3-}\times g_{1-}$, whose contribution to $\WP$ equals
$$
\Btr(z(e^1_1)(\overset\cdot{\rho'}), (\Ad\bD({\la_1}^{-1/2}) - \Ad\bD({\la_1}^{1/2}))\Ad(\rho(e^1_2))(z(e^1_2)(\overset\cdot{\rho''}))).
$$
Here, from Lemma \ref{lem:seamBz}, $z(e^1_1)(\overset\cdot{\rho'})$ and $z(e^1_2)(\overset\cdot{\rho''})$ are in the linear span of $\twomatrix{0}{1}{1}{0}$. But 
$$
\Ad\bD({\la_1}^{-1/2})\twomatrix{0}{1}{1}{0} - \Ad\bD({\la_1}^{1/2})\twomatrix{0}{1}{1}{0}
= \twomatrix{0}{{\la_1}\inv}{{\la_1}}{0} -  \twomatrix{0}{{\la_1}}{{\la_1}\inv}{0} = (\la_1 - {\la_1}\inv)\twomatrix{0}{-1}{1}{0}
$$
is orthogonal to $\twomatrix{0}{1}{1}{0}$ with respect to $\Btr$.
Consequently the contributions of the pair of pants $P = P_j$ to
the Weil-Petersson symplectic form $\WP$ vanishes.\par
This proves Wolpert's formula \eqref{eq:Wformula}.\qed

\section{Lifts to $SL_2(\bR)$}\label{sec:spin}

Finally we discuss lifts of the standard cocycle for a hyperbolic surface to $SL_2(\bR)$. The first two subsections are devoted to a review on spin hyperbolic surfaces, because the author cannot find an expository reference about relations between the spin and hyperbolic structures which fits to our construction.
In this section we denote
$$
D(\la) := \twomatrix{\la}{0}{0}{\la\inv} \in SL_2(\bR)
$$
for any $\la \in \bR \setminus \{0\}$. 

\subsection{Spin structures on an oriented surface}

Now we recall some basic facts on spin structures on a compact oriented surface together with their proofs. \par
Let $\Sigma$ be a compact oriented surface, and 
denote by $U\Sigma$ the unit tangent bundle of the surface $\Sigma$. It is a principal $S^1$ bundle 
defined by the subset $\{e \in T\Sigma; \,\, 
\Vert e\Vert = 1\}$ of the tangent bundle $T\Sigma$ with respect to a Riemannian metric $\Vert\cdot\Vert$, or equivalently the quotient of $T\Sigma \setminus(\text{zero section})$ by scalar multiplication by the positive real numbers $\bR_{>0}$. 
The projection $\varpi: U\Sigma \to \Sigma$ induces the Gysin exact sequence 
\begin{equation}\label{eq:exact2}
0 \to H^1(\Sigma; \bZ/2) \overset{\varpi^*}\to H^1(U\Sigma; \bZ/2) 
\overset{\iota^*}\to (\bZ/2)^{\pi_0(\Sigma)},
\end{equation}
where the map $\iota^*$ is induced by the inclusion of the fiber on each fiber into the total space $U\Sigma$.  
A spin structure on $\Sigma$ is a fiberwise double covering of $U\Sigma$. It can be regarded as an element $s \in H^1(U\Sigma; \bZ/2)$ such that $\iota^*s = (1, 1, \dots, 1) \in (\bZ/2)^{\pi_0(\Sigma)}$, and so defines the $\bmod\, 2$ rotation number $\rot_s$ of any immersed loop $\ell: S^1 \to \Sigma$ by
\begin{equation}\label{eq:rot2}
\rot_s(\ell) := \langle s, [\overset\cdot\ell]\rangle \in \bZ/2,
\end{equation}
where $\overset\cdot\ell: S^1 \to U\Sigma$ is the normalized velocity vector
of $\ell$. It is independent of the orientation of $\ell$, since it is defined over $\bZ/2$. Two spin structures $s$ and $s'$ coincide with each other if $\rot_{s'}(\ell) = \rot_{s}(\ell)$ for any immersed loop $\ell$. In fact, the difference $s'-s$ can be regarded as an element 
of $H^1(\Sigma; \bZ/2)$, and we have $\rot_{s'}(\ell) - \rot_{s}(\ell)
= \langle s'-s, [\ell]\rangle$. 
As was proved by Johnson \cite[Theorem 1A]{J80a}, 
if $\ell$ is a simple closed curve, then the homology class 
$[\overset\cdot\ell] \in H_1(U\Sigma; \bZ/2)$ depends only on the homology class $[\ell] \in H_1(\Sigma; \bZ/2)$. 
\begin{prop}[Poincar\'e-Hopf theorem]\label{prop:PH}
Suppose $\Sigma$ is connected, and its boundary $\pa\Sigma$ is non-empty. Fix a labeling of the boundary components, and denote each of them by $\pa_k\Sigma$, $1 \leq k \leq n$. 
Then there exists a spin structure on $\Sigma$, and we have
$$
\sum^n_{k=1}\rot_s(\pa_k\Sigma) \equiv \chi(\Sigma) \pmod{2}
$$
for any spin structure $s$ on $\Sigma$. 
Here $\chi(\Sigma)$ is the Euler characteristic of the surface $\Sigma$. 
\end{prop}
\begin{proof} Since $\Sigma$ is connected and $\pa\Sigma \neq \emptyset$, the bundle $U\Sigma$ is trivial. Hence there exists a fiberwise double covering of $U\Sigma$. Moreover, for any spin structure $s \in H^1(U\Sigma; \bZ/2)$, 
we have some $\tilde s \in H^1(U\Sigma; \bZ)$ whose $\bmod\, 2$ reduction equals $s \in 
H^1(U\Sigma; \bZ/2)$ and whose restriction to the fiber equals the positive generator in $H^1(S^1; \bZ)$. The cohomology class $\tilde s$ can be regarded as a framing of $U\Sigma$, and the $\bmod \, 2$ reduction of its rotation number $\rot_{\tilde s}$ equals $\rot_s$ for any immersed loop. Hence,
by the Poincar\'e-Hopf theorem, we have 
$$
\sum^n_{k=1}\rot_{\tilde s}(\pa_k\Sigma) = \chi(\Sigma) \in \bZ,
$$
where the orientation on $\pa_k\Sigma$ is induced by that of $\Sigma$. (See, for example, \cite[Lemma 3.3]{Kaw18}.) This proves the proposition. 
\end{proof}

Let $C = \coprod^m_{l=1}C_l$ be a disjoint union of simple closed curves in the surface $\Sigma$. Take an open tubular neighborhood 
$\nu: C\times \mathopen]-2, 2\mathclose[ \hookrightarrow \Sigma$ with $\nu(\cdot, 0) = 1_C$. We denote $\Sigma' := \Sigma \setminus \nu(C\times \mathopen]-1, 1\mathclose[)$ and 
$C_l^\pm := {\nu}(C_l\times \{\pm1\}) \subset \Sigma'$. 
By gluing all $C_l^+$ and $C_l^-$ 
through $\nu(\cdot, 1)$ and $\nu(\cdot, -1)$, 
we have a surjective $C^\infty$ map $q: \Sigma' \to \Sigma$. 
We denote $C' := q\inv(C) = \coprod^m_{l=1} (C_l^+ \amalg C_l^-)
\subset \Sigma'$. 
\begin{lem}\label{lem:spinglue} Let $s'$ be a spin structure on $\Sigma'$. Then $s' = q^*s$ for some spin structure $s$ on $\Sigma$, 
if and only if $\rot_{s'}(C^+_l) = \rot_{s'}(C^-_l)$ for every $1 \leq l \leq m$.
\end{lem}
\begin{proof} In this proof, the coefficients of all the cohomology groups are $\bZ/2$. By excision we have an isomorphism $q^*: H^*(U\Sigma, U\Sigma\vert_{C}) \cong H^*(U\Sigma', U\Sigma'\vert_{C'})$. 
Consider the morphism of cohomology exact sequences of pairs
\begin{equation}\label{eq:spindiag}
\begin{CD}
H^1(U\Sigma, U\Sigma\vert_{C}) 
@>>> H^1(U\Sigma)@>{\iota^*}>> H^1(U\Sigma\vert_{C})
@>>> H^2(U\Sigma, U\Sigma'\vert_{C})\\
@| @V{q^*}VV @V{q^*}VV @| \\
H^1(U\Sigma', U\Sigma'\vert_{C'}) 
@>>> H^1(U\Sigma')
@>{\iota^*}>> H^1(U\Sigma'\vert_{C'})
@>>> H^2(U\Sigma', U\Sigma'\vert_{C'}).
\end{CD}
\end{equation}
The condition on the $\bmod\,2$ rotation numbers is equivalent to 
$\iota^*s'  \in q^*(H^1(U\Sigma\vert_C))$. Clearly it follows from $s' \in q^*(H^1(U\Sigma))$. Conversely, $\iota^*s' \in q^*(H^1(U\Sigma\vert_C))$ implies $s' \in q^*(H^1(U\Sigma))$ by some diagram chasing on \eqref{eq:spindiag}. This proves the lemma.
\end{proof}

As is known, there exists a spin structure also when $\Sigma$ is closed of genus $g$.
This follows from the Gysin exact sequence \eqref{eq:exact2} and the fact $\chi(\Sigma) = 2-2g \equiv 0 \pmod{2}$. We can prove it also as follows: 
Let $\Sigma^c$ be the complement of the interior of an embedded closed disk $D$ in $\Sigma$. Then, by Proposition \ref{prop:PH}, we have spin structures on both of $\Sigma^c$ and $D$, and the $\bmod\,2$ rotation numbers of the boundary $\pa\Sigma^c = \pa D$ for both of them equal $1-2g \equiv 1\pmod{2}$.
Hence, by Lemma \ref{lem:spinglue}, we can glue them to get a spin structure on the closed surface $\Sigma$. \par

\medskip
Now let $s'$ be a spin structure on $\Sigma'$ with $\rot_{s'}(C^+_l) = 
\rot_{s'}(C^-_l)$ for each $1 \leq l \leq m$. By Lemma \ref{lem:spinglue}, the set $\cS(s')$ of all spin structures $s$ on $\Sigma$ with $q^*s = s'$ is non-empty. 
Since the spin structure is defined to be a fiberwise double covering of the unit tangent bundle of the surface, the restriction $U\Sigma\vert_{C_l}$ for each $1 \leq l \leq m$ admits a unique nontrivial covering transformation $\tau_l$ whose restricts to the antipode map of the circle on each fiber. The difference of two $s_1$ and $s_2 \in \cS(s')$ comes from re-glueing the spin structure 
using some of $\tau_l$'s. In other words, the group $H_1(C; \bZ/2) \cong (\bZ/2)^{\oplus n}$ acts on the set $\cS(s')$ transitively 
through the $\tau_l$'s.
Now fix some $1 \leq l_0 \leq m$.
\begin{lem}\label{lem:l_0} If we obtain $s_2 \in \cS(s')$ by re-glueing $s_1\in \cS(s')$ along $C_{l_0}$ using $\tau := \tau_{l_0}$, then the difference $s_2 - s_1 \in H^1(\Sigma; \bZ/2)$ equals the Poincar\'e dual of $[C_{l_0}] \in H_1(\Sigma, \pa\Sigma; \bZ/2)$. 
\end{lem}
\begin{proof} The lemma is equivalent to the equation 
\begin{equation}\label{eq:l_0}
\rot_{s_2}(\ell) - \rot_{s_1}(\ell) = [C_{l_0}] \cdot [\ell] \in \bZ/2
\end{equation}
for any immersed loop $\ell: S^1= \bR/\bZ \to \Sigma$. Here 
the RHS means the $\bmod\,2$ intersection number of the homology classes $[C_{l_0}]$ and $[\ell]$. We may assume $\ell$ is transverse to $C_{l_0}$. It is clear if the intersection $\ell\cap C_{l_0}$ is empty. Hence we can assume $\ell\cap C_{l_0}$ is not empty. Then there exist some real numbers $t_0 < t_1 < \cdots < t_{n-1} < t_n = t_0+1$ with $n \geq 1$ such that $\ell\inv(C_{l_0}) = \{t_i\bmod{\bZ} \in S^1 = \bR/\bZ; \,\, 1 \leq i \leq n\}$. Then we have $n \equiv [C_{l_0}] \cdot [\ell] \pmod{2}$. 
For each $1 \leq i \leq n$, we denote $I_i := [t_{i-1}, t_i] \subset \bR/\bZ$ and take a lift $\ell_i$ of $\overset\cdot\ell\vert_{I_i}: I_i \to U\Sigma$ to the spin structure $s_1$. We have $\ell_{i+1}(t_i) = \tau^{\epsilon_i}(\ell_{i}(t_i))$ for some $\epsilon_i \in \{0,1\}$. 
Then we have $\rot_{s_1}(\ell) \equiv \sum^n_{i=1}\epsilon_i \bmod{2}$. If we consider the spin structure $s_2$ similarly, we have 
$$
\rot_{s_2}(\ell) \equiv \sum^n_{i=1}(\epsilon_i+1) \equiv 
\rot_{s_1}(\ell) + n \equiv \rot_{s_1}(\ell) + [C_{l_0}] \cdot [\ell]
 \bmod{2},
$$
which proves the equation \eqref{eq:l_0} and so the lemma.
\end{proof}
We denote the composite of the homomorphism induced by the inclusion $C \hookrightarrow \Sigma$ and the Poincar\'e duality map by $\imath_*: H_1(C; \bZ/2) \to H_1(\Sigma, \pa\Sigma; \bZ/2) = H^1(\Sigma; \bZ/2)$. By the excision isomorphisms $H_*(C; \bZ/2)
= H_*(C\cup \pa\Sigma, \pa\Sigma; \bZ/2)$ and 
$H_*(\Sigma', \pa\Sigma'; \bZ/2)
= H_*(\Sigma, C\cup \pa\Sigma; \bZ/2)$, we have a long exact sequence
\begin{equation}\label{eq:exactC}
0 \to H_2(\Sigma, \pa\Sigma; \bZ/2) \overset{j_*}\to 
H_2(\Sigma', \pa\Sigma'; \bZ/2) \overset{\pa_*}\to
H_1(C; \bZ/2) \overset{\imath_*}\to H^1(\Sigma; \bZ/2).
\end{equation}

One can deduce the following theorem from what we have proved. 
\begin{thm}\label{thm:spinglue}
Let $s'$ be a spin structure on $\Sigma'$ with $\rot_{s'}(C^+_l) = 
\rot_{s'}(C^-_l)$ for each $1 \leq l \leq m$. Then the action of $H_1(C; \bZ/2)$ on the set $\cS(s')$ is transitive, and its isotropy subgroup at each point equals $\pa_*(H_2(\Sigma', \pa\Sigma'; \bZ/2))$. 
\end{thm}

Here the fundamental classes of the connected components of $\Sigma'$ constitute a $\bZ/2$-basis of the homology group $H_2(\Sigma', \pa\Sigma'; \bZ/2)$. For any connected component
$\Sigma'_0$, $\pa_*[\Sigma'_0] \in H_1(C; \bZ/2)$ is the fundamental class of the boundary of the image $q(\Sigma'_0)$. It should be remarked that two boundary components of $\Sigma'_0$ which have the same image in $C$ cancel each other in $H_1(C; \bZ/2)$.\par
\medskip
Now let $\Sigma_{g,b}$ denote a compact connected oriented surface of genus $g$ with $b$ boundary components. If $\chi(\Sigma_{g,b}) = 2-2g -b$ is negative, it admits a pants decomposition $\{\delta_i\}^{3g-3+b}_{i=1}$, where $\delta_i$ is a simple closed curve in $\Sigma_{g,b}$. It yields $-\chi(\Sigma_{g,b}) = 2g -2+b$ pairs of pants
$\{P_j\}^{2g-2+b}_{j=1}$. Suppose a spin structure is given on each pair of pants such that the $\bmod\, 2$ rotation numbers of each $\delta_i$ for the spin structures in its both sides coincide. Then the set of all the spin structures on $\Sigma_{g,b}$ which restrict to the given one at each pair of pants is bijective to the quotient of $H_1(\bigcup^{3g-3+b}_{i=1}\delta_i; \bZ/2)$ by the action of the group generated by the covering transformations for the given spin structures on the pairs of pants. More explicitly, by changing the label of $\{\delta_i\}$ if necessary, 
we glue all the pair of pants along $\{\delta_i\}^{3g-3+b}_{i=g+1}$ to obtain a connected surface $\hat \Sigma$ of genus $0$, 
which is diffeomorphic to $\Sigma_{0, 2g+b}$.
We can construct the surface $\hat \Sigma$ by adding a pair of pants one by one. 
By Theorem \ref{thm:spinglue} we have a unique spin structure $\hat s$ on $\hat\Sigma$ which restricts to the given spin structure on each pair of pants. 
We have exactly $2^g$ spin structures on $\Sigma_{g,b}$ extending the structure $\hat s$ again by Theorem \ref{thm:spinglue}. In fact, 
the inclusion homomorphism $H_1(\bigcup^g_{i=1}\delta_i; \bZ/2) \to 
H_1(\Sigma_{g,b}, \pa\Sigma_{g,b}; \bZ/2)$ is injective. 
\par

\subsection{Spin structures on a hyperbolic surface}

Since $\PSL$ acts on the unit tangent bundle $U\bH$ of the upper half plane $\bH$ in a free and transitive way, 
the map $\mu_v: \PSL \to U\bH$, $A \mapsto Av$, is a diffeomorphism for any $v \in U\bH$. We denote by $\bp: \SL \to \PSL$ the quotient map, which is a unique (nontrivial) double covering of $\PSL$.  
Let $\Gamma \subset \PSL$ be a discrete subgroup with no elliptic elements. 
The diffeomorphism $\mu_v$ induces a diffeomorphism $\overline{\mu_v}: 
\Gamma\backslash\PSL \to U(\Gamma\backslash\bH)$, thorough which 
a spin structure on $\Gamma\backslash\bH$ is identified with a double covering of the space $\Gamma\backslash\PSL$ whose pullback to $\PSL$ is isomorphic to the double covering $\bp$, since the inclusion map of each fiber $S^1 \hookrightarrow U\bH \cong \PSL$ is a homotopy equivalence.
\par
We call an element $A \in \SL{\setminus \{I\}}$ {\it elliptic} if $\bp(A)$ has a fixed point inside the upper half plane $\bH$. In particular, the matrix $-I$ is elliptic. 
An element $A \in \SL$ is {\it hyperbolic} (resp.\ {\it parabolic}) if $\bp(A) = \pm A \in \PSL$ is hyperbolic (resp.\ parabolic). \par
Let $\hat\Gamma \subset \SL$ be a discrete subgroup with no elliptic elements. Since $-I \not\in \hat \Gamma$, the restriction $\bp\vert_{\hat\Gamma}: \hat\Gamma \to \PSL$ is injective. The image $\Gamma := \bp(\hat\Gamma) \subset \PSL$ is a discrete subgroup with no elliptic elements, so that 
the quotient $\Gamma\backslash\bH$ is a Riemann surface. 
Since $\bp\vert_{\hat\Gamma}: \hat\Gamma \to \Gamma$ is a group isomorphism, the map $\bp$ induces a double covering 
$\bp_{\hat\Gamma}: \hat\Gamma\backslash\SL \to \Gamma\backslash\PSL$, 
$A \bmod{\hat\Gamma} \mapsto \bp(A) \bmod{\Gamma}$. The double covering 
$\overline{\mu_v}\circ \bp\vert_{\hat\Gamma}$ is a spin structure on the surface $\Gamma\backslash\bH$, since its pullback to $\PSL$ equals the covering $\bp$. 
Moreover the spin structure $\overline{\mu_v}\circ \bp\vert_{\hat\Gamma}$ does not depend on the choice of $v \in U\bH$. In order to prove it, we take any other $v' {\in} U\bH$, which equals $v'g$ for some $g \in \PSL$. Then the multiplication by $g$ induces an isomorphism of double coverings
$$
\begin{CD}
\hat\Gamma\backslash\SL @>{\cdot g}>> \hat\Gamma\backslash\SL\\
@V{\overline{\mu_{v'}}\circ\bp_{\hat\Gamma}}VV @V{\overline{\mu_v}\circ\bp_{\hat\Gamma}}VV\\
UT(\Gamma\backslash\bH) @= UT(\Gamma\backslash\bH),
\end{CD}
$$
as was to be proved.\par
We denote the spin structure $\overline{\mu_v}\circ\bp_{\hat\Gamma}$ by $s_{\hat\Gamma}$, and its $\bmod\, 2$ rotation number by $\rot_{\hat\Gamma}$.  
\begin{lem}\label{lem:spinsign}
{\rm (1)} If $A \in \hat\Gamma$ is hyperbolic and 
$C$ its geodesic on $\Gamma\backslash\bH$,
then the sign of $\trace(A)$ equals $(-1)^{\rot_{\hat\Gamma}(C)}$. \par
{\rm (2)} If $A \in \Gamma\setminus\{1\}$ is parabolic and 
$C$ its horocycle on $\Gamma\backslash\bH$,
then the sign of $\trace(A)$ equals $(-1)^{\rot_{\hat\Gamma}(C)}$. 
\end{lem}
\begin{proof} (1) We have $A = PD(\la)P\inv$ for some $P \in \SL$ and 
$\la \in \bR \setminus [-1,1]$, and define $A_t := P\twomatrix{\exp(t\log|\la|)}{0}{0}{\exp(-t\log|\la|)}P\inv \in \SL$ for $t \in [0,1]$. 
The geodesic $C$ is given by $t\in [0,1] \mapsto A_tP\sqrt{-1}\bmod{\Gamma} \in \Gamma\backslash\bH$. The map $t \in [0,1] \mapsto A_t \bmod{\hat\Gamma}$ 
is a lift of its velocity vector with respect to the spin structure 
$\overline{\mu_{v}}\circ\bp_{\hat\Gamma}$, where $v = P(\pa/\pa y)_{\sqrt{-1}}) \in UT_{P\sqrt{-1}}\bH$. Since $A_1 = \frac{\la}{|\la|}A$, 
the lift is a loop, or equivalently $\rot_{\hat\Gamma}(C) = +1$, if and only if
the sign $\frac{\la}{|\la|}$ of $\trace(A)$ equals $+1$. This proves the part (1). \par
(2) We have $A = \epsilon P\twomatrix{1}{\pm 1}{0}{1}P\inv$ for some $\epsilon \in \{\pm1\}$, $P \in \SL$, and define $A_t := P\twomatrix{1}{\pm t}{0}{1}P\inv \in \SL$ for $t \in [0,1]$. A horocycle $C$ for $A$ is given by  $t \in [0,1] \mapsto A_tP(\sqrt{-1}) \bmod{\Gamma} \in \Gamma\backslash\bH$.
The map $t \in [0,1] \mapsto A_t \bmod{\hat\Gamma}$ 
is a lift of its velocity vector with respect to the spin structure 
$\overline{\mu_{v}}\circ\bp_{\hat\Gamma}$, where $v = P(\pa/\pa x)_{\sqrt{-1}}) \in UT_{P\sqrt{-1}}\bH$. Since $A_1 = \epsilon A$, 
the lift is a loop, or equivalently $\rot_{\hat\Gamma}(C) = +1$, if and only if
the sign $\epsilon$ of $\trace(A)$ equals $+1$. This proves the part (2). 
\end{proof}

The following is well-known in the context of super Riemann surfaces \cite{SW19}.
\begin{cor} Let $\hat\Gamma \subset \SL$ be a discrete subgroup with no elliptic points, and suppose the hyperbolic surface $\bp(\hat\Gamma)\backslash\bH$ is of finite volume. Then the number of punctures the trace of whose holonomy are positive is even.
\end{cor}
\begin{proof} Let $g$ be the genus of $\bp(\hat\Gamma)\backslash\bH$, and 
$n_\pm$ the number of punctures the trace of whose holonomy are positive (resp.\ negative). From Proposition \ref{prop:PH} and Lemma \ref{lem:spinsign}, we have $n_- \equiv \chi(\bp(\hat\Gamma)\backslash\bH) = 2-2g-n_+-n_- \pmod{2}$. 
This proves the corollary. 
\end{proof}

Let $\Gamma \subset \PSL$ be a discrete subgroup with no elliptic elements.
A lift of $\Gamma$ is a discrete subgroup $\hat \Gamma \subset \SL$ such that $\bp\vert_{\hat\Gamma}: \hat\Gamma \to \Gamma$ is an isomorphism.
Then $\hat\Gamma$ has no elliptic elements. 
\begin{lem}\label{lem:spinlift} For any spin structure $s$ on 
the surface $\Gamma\backslash\bH$, there exists a unique lift $\hat\Gamma$
of $\Gamma$ such that $s_{\hat\Gamma} = s$. 
\end{lem}
\begin{proof}
(Unigueness) Since $\Gamma$ has no elliptic elements, the lift in $\hat\Gamma$ of any element of $\Gamma$ is uniquely determined by 
the sign of its trace. 
By Lemma \ref{lem:spinsign}, the latter is uniquely determined by the spin structure $s$. This proves the uniqueness.\par
(Existence) 
We denote the pullback of the double covering $s$ of $UT(\Gamma\backslash\bH)$ by the diffeomorphism $\overline{\mu_v}$ 
by $\bp_s: E_s \to \Gamma\backslash\PSL$. 
Choose a point $*\in {\bp_s}\inv(I\bmod{\Gamma})$. 
Recall that a loop $\ell$ which generates the infinite cyclic group $\pi_1(\PSL, I)$
lifts to a path on $\SL$ which connects $I$ to $-I$. Hence the loop 
$\ell \bmod{\Gamma} \in \Gamma\backslash\PSL$ lifts to a path on $E_s$
which connects $*$ to $\tau(*)$, since $s$ is a spin structure. Here $\tau$ is the unique nontrivial covering transformation of $E_s$ as before.\par
The right multiplication of $\PSL$ defines an action 
$m: (\Gamma\backslash\PSL)\times\PSL \to (\Gamma\backslash\PSL)$, $(\bp(A)\bmod{\Gamma}, \bp(B)) \mapsto \bp(AB)\bmod{\Gamma}$. Then there exists 
a unique smooth map $\hat m: (E_s\times \SL, (*, I)) \to (E_s, *)$ such that the diagram 
\begin{equation}\label{eq:liftact}
\begin{xymatrix}{
E_s\times \SL \ar@{-->}[rr]^{\hat m}\ar[d]_{\bp_s\times\bp} 
&&  E_s \ar[d]^{\bp_s}\\
(\Gamma\backslash\PSL)\times\PSL \ar[rr]^-{m} &&\Gamma\backslash\PSL\\
}\end{xymatrix}
\end{equation}
commutes. In order to prove it, we compute the image of $\pi_1(E_s\times \SL, (*, I)) = \pi_1(E_s, *) \times \pi_1(\SL, I)$ by the homomorphism $(m\circ(\bp_s\times\bp))_*$. It is clear that $(m\circ(\bp_s\times\bp))_*\pi_1(E_s, *) = {\bp_s}_*(\pi_1(E_s, *))$. From what we computed above, 
a generator of $\pi_1(\SL, I)$ is mapped to twice the generator of the
fundamental group of the fiber of the projection $\Gamma\backslash\PSL \to \Gamma\backslash\bH$, and so is included in the image
${\bp_s}_*(\pi_1(E_s, *))$. Hence we obtain the existence and uniqueness of the map $\hat m$.\par
By the uniqueness of lifts, one can prove that $\hat m\vert_{E_s\times \{I\}}: (E_s, *) \to (E_s, *)$ equals the identity, and that the two maps $(e, A, B) \mapsto \hat m(e, AB)$ and $(e, A, B) \mapsto \hat m(\hat m(w, A), B)$ from $(E_s\times \SL\times \SL, (*, I, I))$ to $(E_s, *)$ coincide with each other. 
This means $\hat m$ is a right action of $\SL$ on the space $E_s$. 
From what we computed above, we have
\begin{equation}\label{eq:-I}
\hat m(*, -I) = \tau(*) \quad (\neq *).
\end{equation}
The action $\hat m$ is transitive. In fact, for any $e\in E_s$, there exists an element $A \in \SL$ with $\hat m(*, A) \in \{e, \tau(e)\}$, since the action of $\PSL$ on $\Gamma\backslash\PSL$ is transitive. Exchanging $A$ with $-A$ by \eqref{eq:-I} if necessary, we have $\hat m(*, A) = e$.\par
Let $\hat\Gamma \subset \SL$ be the isotropy subgroup 
at the point $* \in E_s$ with respect to the action $\hat m$. 
Since the action $\hat m$ is transitive, we have $\hat\Gamma\backslash\SL \cong E_s$. Moreover $-I \not\in \hat\Gamma$ by \eqref{eq:-I}.\par
From \eqref{eq:liftact}, we have $\bp(A) \in \Gamma$ for any $A \in \hat\Gamma$. By the commutative diagram \eqref{eq:liftact}, 
if $B \in \SL$ satisfies $\bp(B) \in \Gamma$, then $\hat m(*, \pm B) \in \{*, \tau(*)\}$.
From \eqref{eq:-I}, one of $\pm B$ is mapped to $*$ by $\hat m(*, -)$, and so 
is in $\hat \Gamma$. Hence $\bp\vert_{\hat\Gamma}: \hat\Gamma \to \Gamma$ 
is a group isomorphism. In particular, $\hat\Gamma \subset \SL$ is a discrete subgroup with no elliptic elements, and a lift of $\Gamma$. \par
Consequently we obtain a double covering $\hat\Gamma\backslash\SL \cong E_s \to \Gamma\backslash\PSL \cong UT(\Gamma\backslash\bH)$, $A\bmod{\hat\Gamma} \mapsto \hat m(*, A) \mapsto \bp(A)\bmod{\Gamma} \mapsto \bp(A)v \bmod{\Gamma}$, which is isomorphic to both of $\overline{\mu_v}\circ\bp_{\hat\Gamma}$ and the given spin structure $s$.
Hence we obtain $s = s_{\hat\Gamma}$, as was to be shown. \end{proof}

\subsection{Standard cocycles for a spin hyperbolic structure}

For any groupoid cocycle $\rho: \Pi P\op\vert_{P^{(0)}} \to SL_2(\bR)$, we denote by $\bar\rho: \Pi P\op\vert_{P^{(0)}}
\to PSL_2(\bR)$ its reduction by the quotient map $\bp$.
We call a flat $SL_2(\bR)$-bundle over a compact surface {\it Fuchsian}, if its reduction to $PSL_2(\bR)$ is Fuchsian with geodesic boundary. Then let $\la_k$ and ${\la_k}^{-1}$ with 
$|\la_k| > 1$ be the eigenvalues of the $SL_2(\bR)$-part of $\rho(e^1_{k0}e^1_{k1})$. We denote
$\epsilon_k := \la_k/|\la_k| \in \{\pm1\}$ and 
$\mu_k := |\la_k|^{1/2} > 0$. Then we have $\la_k = \epsilon_k{\mu_k}^2$. 
By Proposition \ref{prop:PH} or a lemma of Keen \cite[Lemma 1, p.210]{Keen1}, i.e., Lemma \ref{lem:Keen}, 
we have 
\begin{equation}\label{eq:spinP}
\epsilon_0\epsilon_1\epsilon_2 = -1. 
\end{equation}
\par
Now suppose the reduction $\bar\rho$ is the standard cocycle
for the associated hyperbolic structure. Then, by Theorem \ref{thm:A_k}, we have 
$\bar\rho(e^1_k) = \pm\twomatrix{a_k}{b_k}{-b_k}{-a_k} \in PSL_2(\bR)$ for some $a_k, b_k \in \bR$ with $b_k > a_k > 0$. Explicitly we have $a_k = \sqrt{|b_kc_k|-1}$ and $b_k = |b_kc_k|^{1/2}$.
\begin{lem}\label{lem:signhex} If $\tilde A_k = \twomatrix{a_k}{b_k}{-b_k}{-a_k} \in SL_2(\bR)$, then we have
$$
\aligned
& D({\mu_2}^{-1})\tilde A_2D({\mu_1}^{-1})\tilde A_1D({\mu_0}^{-1})\tilde A_0 = I,\\
& D({\mu_2})\tilde A_2D({\mu_1})\tilde A_1D({\mu_0})\tilde A_0 = -I.
\endaligned
$$
\end{lem}
\begin{proof} 
Since $\mu_1 > 1$ and $b_k > a_k > 0$, we have 
\begin{equation}\label{eq:muab}
\mu_1b_1b_2 > b_1b_2 {>} a_1a_2 > {\mu_1}^{-1}a_1a_2.
\end{equation}
The LHS's of both of the above equations equal 
$\pm I$, since each of them is the monodromy of $\bar\rho$ along the boundary of a hexagon. 
Now we compute 
$$
\aligned
& D({\mu_2}^{-1})\tilde A_2D({\mu_1}^{-1})\tilde A_1 
= \twomatrix{{\mu_2}^{-1}a_2}{{\mu_2}^{-1}b_2}{-{\mu_2}b_2}{-{\mu_2}a_2}\twomatrix{{\mu_1}^{-1}a_1}{{\mu_1}^{-1}b_1}{-{\mu_1}b_1}{-{\mu_1}a_1}\\
&= \twomatrix{{\mu_2}^{-1}({\mu_1}^{-1}a_1a_2 - \mu_1b_1b_2)}{*}{*}{*},
\endaligned
$$
whose $(1,1)$-entry is negative from \eqref{eq:muab}. 
Since $1 = \det \tilde A_k = -{a_k}^2 + {b_k}^2$, we have ${\tilde A_k}^2 = -I$, and so ${\tilde A_k}^{-1} = -\tilde A_k$. The $(1,1)$-entry of $(D({\mu_0}^{-1})\tilde A_0)^{-1} = -\tilde A_0D({\mu_0})$ is also negative. This proves the first equation. \par
Similarly we compute 
$$
\aligned
& D({\mu_2})\tilde A_2D({\mu_1})\tilde A_1 
= \twomatrix{{\mu_2}a_2}{{\mu_2}b_2}{-{\mu_2}^{-1}b_2}{-{\mu_2}^{-1}a_2}\twomatrix{{\mu_1}a_1}{{\mu_1}b_1}{-{\mu_1}^{-1}b_1}{-{\mu_1}^{-1}a_1}\\
&= \twomatrix{*}{*}{*}{{\mu_2}^{-1}({\mu_1}^{-1}a_1a_2 - \mu_1b_1b_2)},
\endaligned
$$
whose $(2,2)$-entry is negative. The $(2,2)$-entry of $(D({\mu_0})\tilde A_0)^{-1} = -\tilde A_0D({\mu_0}^{-1})$ is positive. This proves the second equation. 
\end{proof}
This lemma is compatible with the equation \eqref{eq:spinP}.

\begin{prop}\label{prop:sdstd} 
For any Fuchsian flat $SL_2(\bR)$-bundle over $P$, there exists a unique groupoid cocycle $\rho: \Pi P\op\vert_{P^{(0)}} \to SL_2(\bR)$ of the flat bundle 
satisfying the following conditions, which we call it the {\rm standard cocycle} for the flat bundle. 
\par
{\rm (i)} Its reduction $\bar\rho$ is the standard cocycle
for the associated hyperbolic structure.\par
{\rm (ii)} The $(1,1)$-entry of $\rho(e^1_k)$ is positive for any $k \in \bZ/3$.\par
{\rm (iii)} The $(1,1)$-entry of $\rho(e^1_{k0})$ is positive for any $k \in \bZ/3$.
\par
\noindent
Moreover a gauge transformation $b: P^{(0)} \to SL_2(\bR)$ preserves the standard cocycle if and only if $b$ is a constant function with values in $\pm I$. 
\end{prop}
\begin{proof} 
We may choose a holonomy $\rho$ whose reduction is the standard cocycle. 
Then we have $\rho(e^1_{k,0}) = \epsilon'_kD(\mu_k)$, 
$\rho(e^1_{k,1}) = \epsilon_k\epsilon'_kD(\mu_k)$, and
$\rho(e^1_k) = \epsilon''_k\tilde A_k$ for some $\epsilon'_k, \epsilon''_k \in \{\pm1\}$. Here we have
\begin{equation}\label{eq:aepsilon}
\epsilon'_0\epsilon'_1\epsilon'_2\epsilon''_0\epsilon''_1\epsilon''_2 = 1
\end{equation}
from Lemma \ref{lem:signhex}.\par
A gauge transformation $b: P^{(0)} \to SL_2(\bR)$ preserves the condition (i) if and only if $b(e^0_{k,0}) = D(\eta_{k,0})$ and $b(e^0_{k,1}) = D(\eta_{k,1})$ for some $\eta_{k,0}, \eta_{k,1} \in \{\pm1\}$. Then we have 
$$
\aligned
&(\rho^{b}
)(e^1_k) = \eta_{k,0}\eta_{k-1,1}\epsilon''_k\tilde A_k,\\
&(\rho^{b}
)(e^1_{k,0}) = \eta_{k,0}\eta_{k,1}\epsilon'_kD(\mu_k).
\endaligned
$$
Hence $\rho^{b}
$ satisfies the conditions (ii) and (iii) if and only if 
$$
\aligned
& \eta_{k,0}\eta_{k-1,1} = \epsilon''_k,\\
& \eta_{k,0}\eta_{k,1} = \epsilon'_k
\endaligned
$$
for any $k = 0,1,2$. By the condition \eqref{eq:aepsilon}, the equations have exactly two solutions. 
This proves the lemma. 
\end{proof}

\noindent
The constant function $P^{(0)} \to \SL$ with values in $-I$ corresponds to the covering transformation $\tau$. \par
\medskip
Next we fix a pants decomposition 
$\cP = \{\delta_i\}^{3g-3}_{i=1}$ of the surface $\Sg$ to get 
$2g-2$ pairs of pants $\{P_j\}^{2g-2}_{j_1}$. The decomposition $\cP$ 
defines a cell decomposition $\Sigma_{\cP}$ as in the previous sections.
For any Fuchsian flat $SL_2(\bR)$-bundle over $\Sg$, we can take a groupoid cocycle $\rho: \Pi \Sigma\op_{\cP}\vert_{(\Sigma_{\cP})^{(0)}} \to SL_2(\bR)$, whose restriction to $P_j$ equals the standard cocycle on $P_j$ for any $1 \leq j \leq 2g-2$.
In general, let  $\la \neq 0, \pm 1$. If a matrix $A = \twomatrix{a}{b}{c}{d} \in \SL$ satisfies $D(\la)AD(\la) = A$, then we have $a = d = 0$ and $bc = -1$. 
Hence, for any $1 \leq i \leq 3g-3$, we have $\rho(e^1_{((i))0}) = \rho(e^1_{((i))1}) = \twomatrix{0}{-{T_i}\inv}{T_i}{0}$ for some $T_i \neq 0$. Thus we obtain a coordinate $(\la_i, T_i)^{3g-3}_{i=1}
\in ((\bR\setminus [-1,1])\times (\bR\setminus\{0\}))^{3g-3}
$ of the spin Teichm\"uller space.
The coordinate $(\la_i, T_i)^{3g-3}_{i=1}$ can bee regarded as 
an element of the space $(\bR\setminus [-1,1])^{\cP}\times (\bR\setminus\{0\})^{\cP}$. 
\par
For each $1 \leq j \leq 2g-2$, the standard cocycle on $P_j$ is preserved only by a gauge transformation $P^{(0)}_j \to \{\pm I\}$.
We denote by $b_j: P^{(0)}_j \to \{-I\} \subset \SL$ the unique nontrivial transformation. These transformations generate an abelian group $\cB := \langle b_j; \,\, 1 \leq j \leq 2g-2\rangle \cong (\bZ/2)^{2g-2}$, and act on the $T$-part of the coordinates $T \in (\bR\setminus\{0\})^{\cP}$ by 
$$
(b_j\cdot T)(\delta_i) = 
\begin{cases}
-T(\delta_i), & \text{if $\delta_i$ is a boundary component of the subsurface given by $P_j$}, \\
T(\delta_i), & \text{otherwise}.
\end{cases}
$$
Here, if two boundary components of $P_j$ are the same $\delta_i$, 
then we have $(b_j\cdot T)(\delta_i) = T(\delta_i)$.
If we denote $b_0 := b_1b_2\cdots b_{2g-2}$, which is the unique nontrivial covering transformation of the spin structure on the whole $\Sg$, then the quotient group $\cB/\langle b_0\rangle \cong (\bZ/2)^{2g-3}$ acts on the space $(\bR\setminus\{0\})^{\cP}$ freely.
Let the group $\cB$ act on the $\la$-part $(\bR\setminus [-1,1])^{\cP}$ trivially. Then the spin Teichm\"uller space for the surface $\Sg$ is homeomorphic to the quotient 
$((\bR\setminus [-1,1])^{\cP}\times (\bR\setminus\{0\})^{\cP})/\cB$.
\par
As in \S4.1, by changing the label of $\{\delta_i\}$ if necessary, 
we glue all the pair of pants along $\cP' := \{\delta_i\}^{3g-3}_{i=g+1}$ to obtain a connected surface $\hat \Sigma$ diffeomorphic to $\Sigma_{0, 2g}$. Then, under the action of the group $\cB$, 
we may take all $T_i$'s for $i \geq g+1$ {\it positive}. 
Thus we obtain the following.
\begin{prop}\label{prop:SLstd} For any Fuchsian flat $\SL$-bundle over $\Sg$,
there exists a unique groupoid cocycle $\rho: \Pi \Sigma\op_{\cP}\vert_{(\Sigma_{\cP})^{(0)}} \to SL_2(\bR)$ satisfying the following conditions, which we call it {\rm the standard cocycle} for the flat bundle on $\Sigma_{\cP}$.\par
{\rm (i)} The restriction of $\rho$ to the pair $P_j$ of pants equals the standard cocycle on $P_j$ for any $1 \leq j \leq 2g-2$.
\par
{\rm (ii)} For any $g+1 \leq i \leq 3g+3$, we have $\rho(e^1_{((i))0}) = \epsilon_i\rho(e^1_{((i))1}) = \twomatrix{0}{-{T_i}\inv}{T_i}{0}$ for some $T_i > 0$.
\end{prop}

In particular, for $1 \leq i \leq g$, we have $\rho(e^1_{((i))0}) = \rho(e^1_{((i))1}) = \twomatrix{0}{-{T_i}\inv}{T_i}{0}$ for some $T_i \neq 0$.
By Theorem \ref{thm:twistparam}, $\tau_i := -2\log|T_i|$ is the twisting parameter along $\delta_i$ for any $1 \leq i \leq 3g-3$. 
The set of signs $\{T_i/|T_i|\}^g_{i=1}$ indicates a choice among the $2^g$ spin structures on $\Sg$, as was stated in the end of \S4.1.
\par

\section{{Appendix: The Shimura isomorphism and symplectic forms on the Teichm\"uller space}}
Up to some multiplicative constant, the Weil-Petersson symplectic form $\WP$ in Wolpert \cite{W2} equals the Atiyah-Bott-Goldman symplectic form defined as a cup product on the twisted cohomology group in Goldman \cite{G1}. On the Lie algebra $\sltwo$, we can take the Killing form $B$ and the trace form $\Btr$ of the standard representation on $\bR^2$ as invariant non-degenerate symmetric forms. We denote by $\ABG$ and $\ABGtr$ the Atiyah-Bott-Goldman symplectic form using $B$ and $\Btr$, respectively. Since $B = 4\Btr$, we have $\ABG = 4\ABGtr$. It is $\ABGtr$ that we compute in this paper. In this section, by reading the paper \cite{G1} in detail, and comparing it with Wolpert's works \cite{W1, W2},
we prove 
\begin{prop}\label{prop:WPABG}
$$
2\WP = \ABG = 4\ABGtr.
$$
\end{prop}
\noindent
The key ingredient of the proof is Wolpert's argument \cite[\S4]{W1}
based on Hejhal's result \cite{H}.\par
Let $z$ be the standard coordinate of the complex line $\bC$, 
which is divided into three parts: $\bH = \{z \in \bC; \,\, \Im z > 0\}$, 
$\bR$ and $\bL := \{z \in \bC; \,\, \Im z < 0\}$. Here we denote $i = \sqrt{-1}$, $y = \Im z = (z-\zb)/2i$ and $\bC P^1 = \hat\bC = \bC\cup \{\infty\}$ in order to refer some fundamental facts on quasi-conformal mappings \cite{A}.
Let $\Gamma \subset \PSL$ be a Fuchsian group such that $M = \bH/\Gamma$ is a closed Riemann surface of genus $g$, and 
$\mathcal{F} \subset \bH$ a fundamental domain of the action of $\Gamma$. The tangent space of the Teichm\"uller space $\mathcal{T}_g$ at a point over the isomorphism class of $M$ is regarded as the cohomology group $H^1(M; K^{\otimes (-1)})$. 
Here $K = K_M$ is the holomorphic cotangent bundle of the Riemann surface $M$. 
Its Serre dual is the space $H^0(M; K^{\otimes 2})$ of holomorphic quadratic differentials on $M$, which is identified with the space of $\Gamma$-invariant holomorphic quadratic differentials on $\bH$. 
\par

\subsection{The Weil-Petersson metric and the Shimura isomorphism}

The Weil-Petersson Hermitian metric on the Teich\"uller space $\mathcal{T}_g$ is induced by a Hermitian inner product on the space $H^0(M; K^{\otimes 2})$ coming from the hyperbolic metric $dA = (\Im z)^{-2}dxdy = -2i(z-\zb)^{-2}dz\wedge d\zb$. For two $\Gamma$-invariant holomorphic quadratic differrentials $\Xi = \xi(z)dz^{\otimes 2}$ and $\Psi = \psi(z) dz^{\otimes 2}$ on $\bH$, Goldman \cite[p.212]{G1} and Wolpert \cite[p.210]{W2} define the same Hermitian product 
$$
h(\Xi, \Psi) = \langle \Xi, \Psi\rangle = \int_{\mathcal{F}}\xi(z)\overline{\psi}(z) (\Im z)^{2}dxdy = -\frac14\int_{\mathcal{F}}\Xi\cdot (dA)\inv\cdot\overline{\Psi}.
$$
Here we mean by the symbol $\langle \Xi, \Psi\rangle$ the Hermitian product in \cite{G1}, though it means $\frac12\Re h(\Xi, \Psi)$ in \cite{W2}.  In this section we denote
$$
\hat h(\Xi) := (\Im z)^2\overline{\xi(z)}d\zb\otimes ({\pa}/{\pa z})
$$
for any $\Gamma$-invariant holomorphic quadratic differrential $\Xi = \xi(z)dz^{\otimes 2} \in H^0(M; K^{\otimes 2})$. The image $\hat h(H^0(M; K^{\otimes 2}))$ is called the space of harmonic Beltrami differentials. Taking their Dolbeault cohomomology classes, we get an anti-linear isomorphism $H^0(M; K^{\otimes 2}) \to H^1(M; K^{\otimes (-1)})$, $\Xi \mapsto [\hat h(\Xi)]$. In \cite[p.228]{W2} the Weil-Petersson Hermitian product on the tangent space is given by 
$$
h([\hat h(\Xi)], [\hat h(\Psi)])  
= \int_{\mathcal{F}}\overline{\xi(z)}\psi(z)(\Im z)^{-2}dxdy
= \overline{h(\Xi, \Psi)}
$$
for any $\Xi$ and $\Psi \in H^0(M; K^{\otimes 2})$. 
Hence we adopt the anti-linear isomorphism $\hat h$ as in \cite{W2}, which is different from that in \cite[p.212]{G1}. 
The Weil-Petersson symplectic form $\WP$ is given by 
\begin{equation}\label{eq:dfnwP}
\WP([\hat h(\Xi)], [\hat h(\Psi)]) = -2\Im h([\hat h(\Xi)], [\hat h(\Psi)]) 
= 2 \Im h(\Xi, \Psi)
\end{equation}
in \cite[p.228]{W2}.
\par
Next we consider the column vector field $Z = \begin{pmatrix} z^2\pa/\pa z \\ z\pa/\pa z \\ \pa/\pa z \end{pmatrix}$ in \cite[\S2.3]{G1} in detail. In order to make the situation clear, we use the classical basis
$$
E = \twomatrix{0}{1}{0}{0}, \quad
H = \twomatrix{1}{0}{0}{-1}, \quad
F = \twomatrix{0}{0}{1}{0}
$$
of $\sltwoc$. The Killing form $B$ of $\sltwoc$ is represented by the matrix
$$
\begin{pmatrix}0 & 0 & 4\\ 0 & 8 & 0 \\ 4 & 0 & 0
\end{pmatrix}
$$
with respect to the basis $\{E, H, F\}$. 
It induces an isomorphism $\tilde B: \sltwoc^* \to \sltwoc$. We denote the dual basis of $\{E, H, F\}$ by 
$\{E^*, H^*, F^*\}$. Then we have $\tilde B(E^*) = \frac14F$, 
$\tilde B(H^*) = \frac18F$ and $\tilde B(F^*) = \frac14E$. 
The Killing form $B$ naturally defines a symmetric pairing $B^*$
by $B^*(-,-) = B(\tilde B(-), \tilde B(-))$, which is represented by the matrix
$$
\begin{pmatrix}0 & 0 & 1/4\\ 0 & 1/8 & 0 \\ 1/4 & 0 & 0
\end{pmatrix}
$$
with respect to the basis $\{E^*, H^*, F^*\}$. 
The action of $PSL_2(\bC)$ on the Riemann sphere $\bC P^1$ defines a Lie algebra isomorphism $\imath: \sltwoc \to H^0(\hat\bC; T\hat\bC)$. One can compute 
$$
\imath(pE+qH+rF)= (-rz^2+2qz+p)\dfrac{\pa}{\pa z}, 
$$
which means $Z = (-z^2F^* + 2zH^* + E^*)(\pa/\pa z)$
from the original definition of the column vector $Z$ in the third paragraph of \cite[2.3]{G1}. Hence we obtain an $SL_2(\bC)$-invariant 
$\sltwoc_{\mathrm{Ad}}$-valued holomorphic vector field
$$
\aligned
\tilde B(Z) 
& = (-\frac14z^2E+\frac14zH+\frac14F)(\pa/\pa z)\\
&= -\frac14\twomatrix{-z}{z^2}{-1}{z}\otimes \left(\dfrac{\pa}{\pa z}\right)
\in H^0(\hat\bC; \sltwoc_{\mathrm{Ad}}\otimes T\hat\bC).
\endaligned
$$
Since the $1$-form $Z$ 
has values in $\sltwoc^*$, so we compute 
$$
\aligned
&B^*(Z\wedge\overline{Z})\\
&= B((-z^2F^* + 2zH^* + E^*), (-z^2F^* + 2zH^* + E^*))(\pa/\pa z)\wedge (\pa/\pa \zb)\\
&= (-\frac1{4}z^2+\frac2{4}z\zb-\frac1{4}\zb^2)(\pa/\pa z)\wedge (\pa/\pa \zb) = -\frac14(z-\zb)^2(\pa/\pa z)\wedge (\pa/\pa \zb)\\
\endaligned
$$
Hence we have
$
B^*((Z\cdot\Xi)\wedge(\overline{Z}\cdot\overline{\Psi}))
= -2i\xi(z)\overline{\psi(z)}y^2dx\wedge dy,
$
and so 
\begin{equation}\label{eq:pairB}
\aligned
\langle \Xi, \Psi\rangle &= \frac{i}{2}\int_MB^*((Z\cdot\Xi)\wedge(\overline{Z}\cdot\overline{\Psi}))\\
&= \frac{i}{2}\int_MB((\tilde B(Z)\cdot\Xi)\wedge(\overline{\tilde B(Z)}\cdot\overline{\Psi}))\\
&= 2i\int_M\trace((\tilde B(Z)\cdot\Xi)\wedge(\overline{\tilde B(Z)}\cdot\overline{\Psi})),
\endaligned
\end{equation}
since $B^*(-,-) = B(\tilde B(-),\tilde B(-)) 
= 4\trace(\tilde B(-)\tilde B(-))$.

The Shimura isomorphism $H^0(M; K^{\otimes 2}) \to H^1(\Gamma; \sltwo_{\mathrm{Ad}})$ maps $\Xi = \xi(z)dz^{\otimes 2} \in 
H^0(M; K^2)$ to the cohomology class of the group cocycle
\begin{equation}\label{eq:cShi}
u'_{\Xi}: A \in \Gamma \mapsto -\frac14\int^{Az_0}_{z_0}\xi(z)\twomatrix{-z}{z^2}{-1}{z} dz = \int^{Az_0}_{z_0}\tilde B(Z)\cdot\Xi
\in \sltwoc.
\end{equation}
From Proposition in \cite[2.5]{G1} we obtain
$$
h(\Xi, \Psi) = \langle \Xi, \Psi\rangle = \frac{i}{2}B_*\left(
[u'_\Xi] \cup \overline{[u'_\Psi]}\right)\cdot [M].
$$
By \eqref{eq:dfnwP} this implies 
\begin{equation}\label{eq:ShiWP}
\WP(\hat h(\Xi), \hat h(\Psi)) = \Re B_*\left(
[u'_\Xi] \cup \overline{[u'_\Psi]}\right)\cdot [M].
\end{equation}
\par

\subsection{Quasi-conformal mappings}
Next we recall some fundamental facts on quasi-conformal mappings stated in \cite{A} to understand the last sentence of \cite[2.4]{G1} `It has been shown by various authors...' in detail. 
We denote by $\mu(f) := f_{\zb}/f_z$ the complex dilatation of a quasi-conformal mapping $f: \hat\bC \to \hat\bC$. 
Let $g$ and $h: \hat\bC \to \hat\bC$ be quasi-conformal mappings. 
Then the chain rule implies 
\begin{equation}\label{eq:ChainR}
\mu(h)\circ g = 
\frac{\mu(h\circ g) - \mu(g)}{1 - \overline{\mu(g)}\mu(h\circ g)}\,
\frac{g_z}{\overline{g}_{\zb}}
\end{equation}
\cite[p.9, (6)]{A}.  If we define 
$$
A^*\mu := (\mu\circ A)\frac{\overline{A'}}{A'}
$$
for $A \in PSL_2(\bC)$ and $\mu \in L^\infty(\hat\bC)$,
we have $\mu(f\circ A) = A^*(\mu(f))$ and
$$
\mu(f\circ A\circ f\inv)\circ f = \frac{A^*(\mu(f)) - \mu(f)}{1 - \overline{\mu(f)}A^*(\mu(f))}\, \frac{f_z}{\overline{f}_{\zb}}
$$
by \eqref{eq:ChainR}. 
Hence $f\circ A\circ f\inv \in PSL_2(\bC)$ is equivalent to $A^*(\mu(f)) = \mu(f)$. \par
Suppose $\nu \in L^\infty(\hat\bC)$ satisfies $A^*\nu = \nu$. 
For $|t|\ll 1$, there is a unique quasi-conformal mapping $f^{t\nu}: \hat\bC \to \hat\bC$ which fixes $0$, $1$ and $\infty$ pointwise, and satisfies $\mu(f^{t\nu}) = t\nu$. We have $A^{t\nu}:= f^{t\nu}\circ A\circ (f^{t\nu})\inv \in PSL_2(\bC)$. Similarly to \cite[p.77]{A}, we denote 
$$
\overset\cdot{F}[\nu] := \frac{d}{dt}\Bigr\vert_{t=0}F(t\nu)
$$
for a function $F(t\nu)$ of $t\nu$. 
We have $\overset\cdot{A}[\nu]A\inv \in \sltwoc$. 
Some straight-forward computation shows that 
$$
\frac{d}{dt}\Bigr\vert_{t=0}(f^{t\nu})\inv(z) = - \overset\cdot{f}[\nu](z),
$$
and so that
$$
\overset\cdot{A}[\nu](z) = \overset\cdot{f}[\nu](Az) - (A'(z))\overset\cdot{f}[\nu](z).
$$
Since $A'\circ A\inv = ((A\inv)')\inv$ by the chain rule, 
we obtain 
$$
\overset\cdot{A}[\nu](A\inv z) = \overset\cdot{f}[\nu](z)
- \left(((A\inv)')\inv\cdot (\overset\cdot{f}[\nu]\circ A\inv)\right)(z),
$$
or equivalently
\begin{equation}\label{eq:fcobound}
\imath\left(\overset\cdot{A}[\nu]\circ A\inv\right)
= \overset\cdot{f}[\nu]({\pa}/{\pa z}) - (A\inv)^*\left(\overset\cdot{f}[\nu]({\pa}/{\pa z})\right).
\end{equation}
On the other hand, as was shown in \cite[p.60]{A}, we have 
\begin{equation}\label{eq:varorig}
\overset\cdot{f}[\nu](\zeta) = -\frac1\pi\int\hskip-7pt\int_{\bC}\nu(z)R(z, \zeta)dxdy,
\end{equation}
where
$$
R(z,\zeta) = \frac1{z-\zeta} - \frac\zeta{z-1} +\frac{\zeta-1}{z}
= \frac{\zeta(\zeta-1)}{z(z-1)(z-\zeta)}.
$$
\par
Now let $\la(z) ({\pa}/{\pa z})\otimes d\zb$, $z \in \bH$, be a $\Gamma$-invariant Beltrami differential, whose Dolbealt class $[\la(z) ({\pa}/{\pa z})\otimes d\zb]$ is a tangent vector of the Teichm\"uller space $\mathcal{T}_g$ at a point corresponding to $\Gamma$. 
We define 
$$
\hat\la(z) := \begin{cases}
\la(z), & \text{if $z \in \bH$},\\
0, & \text{if $\Im z = 0$},\\
\overline{\la(\zb)}, & \text{if $z \in \mathbb{L}$},
\end{cases}
\quad\text{and}\quad
\tilde\la(z) := \begin{cases}
\la(z), & \text{if $z \in \bH$},\\
0, & \text{if $\Im z \leq 0$}.
\end{cases}
$$
Since $f^{t\hat\la}(\bH) = \bH$ for any $|t|\ll 1$,  we have $f^{t\hat\la}\Gamma(f^{t\hat\la})\inv \subset \PSL$, which equals the deformation of the Fuchsian group $\Gamma$ along the path $t\la$.
Hence the map
$$
\hat c(\la): A \in \Gamma \mapsto \overset\cdot{A}[\hat\la]A\inv \in \sltwo
$$
is a well-defined cocycle, whose cohomology class $[\hat c(\la)] \in H^1(\Gamma; \sltwo_{\mathrm{Ad}})$ is identified with the tangent vector $[\la]$. Hence the Atiyah-Bott-Goldmann symplectic form $\ABG$ is given by 
\begin{equation}\label{eq:ABG}
\ABG([\la_1(z)({\pa}/{\pa z})\otimes d\zb], [\la_2(z)({\pa}/{\pa z})\otimes d\zb]) = B_*([\hat c(\la_1)]\cup [\hat c(\la_2)])\cdot [M]
\end{equation}
for any $\Gamma$-invariant Beltrami differentials $\la_1({\pa}/{\pa z})\otimes d\zb$ and $\la_2({\pa}/{\pa z})\otimes d\zb$. 
\par
The Shimura isomorphism is directly related to 
the quasi-Fuchsian group $f^{t\tilde\la}\Gamma(f^{t\tilde\la})\inv \subset PSL_2(\bC)$. The cohomology class $[\tilde c(\la)] \in H^1(\Gamma; \sltwoc_{\mathrm{Ad}}) = H^1(\Gamma; \sltwo_{\mathrm{Ad}})\otimes_{\bR}\bC$ defined by the cocycle
$$
\tilde c(\la): A \in \Gamma \mapsto \overset\cdot{A}[\tilde\la]A\inv \in \sltwoc
$$
satisfies 
\begin{equation}\label{eq:realc}
[\hat c(\la)] = 2\Re [\tilde c(\la)] \in H^1(\Gamma; \sltwo_{\mathrm{Ad}}).
\end{equation}
In fact, we have 
$
\overset\cdot{f}[\hat\la](\zeta)
= \overset\cdot{f}[\tilde\la](\zeta) + \overline{\overset\cdot{f}[\tilde\la](\overline{\zeta})}
$
from \eqref{eq:varorig}, and so 
$
\hat c(\la) = \tilde c(\la) + \overline{\tilde c(\la)}
$ 
from \eqref{eq:fcobound}.
\par
Based on Hejhal's result \cite{H}, Wolpert \cite{W1} relates the cohomology class $[\tilde c(\la)] \in H^1(\Gamma; \sltwoc_{\mathrm{Ad}})$ to the Shimura isomorphism.
It should be remarked that $\overline{\xi(\zb)}dz^{\otimes 2}$, $z \in\bL$, is a $\Gamma$-holomorphic quadratic differential for any $\Xi = \xi(z)dz^{\otimes 2} \in H^0(M; K^{\otimes 2})$.
We set $\la(z) = \la_\Xi(z) = (\Im z)^2\overline{\xi(z)}$, $z \in \bH$, i.e., 
$\la_\Xi(z)d\zb\otimes ({\pa}/{\pa z}) = \hat h(\Xi)$.
Then we have $[\hat h(\Xi)] = \hat c(\la_\Xi)$ as tangent vectors of the Teichm\"uller space.
For $|t|\ll 1$, $f^{t\tilde\la}$ equals $w^t$ for $\varphi =0$ and $\psi(z) = -\frac12\overline{\xi(\zb)}$ in \cite{W1}. Fix a point $z_0 \in \bH$. 
As was deduced by Wolpert \cite[\S4]{W1} from Hejhal's result \cite{H},
there is a path $N^t \in PSL_2(\bC)$ for $|t|\ll 1$ such that 
$$
((N^0)\inv N^t)w^tA(w^t)\inv((N^0)\inv N^t)\inv = A - \frac{t}{4}\int^{A\overline{z_0}}_{\overline{z_0}}\overline{\xi(\zb)}\twomatrix{-z}{z^2}{-1}{z}A dz + O(t^2)
$$
for any $A \in \Gamma$. Hence, if we denote $X := \dfrac{d}{dt}\Bigr\vert_{t=0}(N^0)\inv N^t \in \sltwoc$, then we have 
$$
X + \overset\cdot{A}[\tilde\la]A\inv - AXA\inv
= -\frac14\int^{A\overline{z_0}}_{\overline{z_0}}\overline{\xi(\zb)}\twomatrix{-z}{z^2}{-1}{z} dz
= -\frac14\overline{\int^{Az_0}_{z_0}\xi(z)\twomatrix{-z}{z^2}{-1}{z} dz}.
$$
In other words, the complex conjugate of the cocycle $u'_{\Xi}$ 
\eqref{eq:cShi} is cohomologous to $\tilde c(\la)$ for $\la(z) = (\Im z)^2\overline{\xi(z)}$. 
Hence, from \eqref{eq:realc}, we obtain 
\begin{equation}\label{eq:realu}
[\hat c(\la_\Xi)] = 2\Re [u'_{\Xi}] \in H^1(\Gamma; \sltwo_{\mathrm{Ad}}).
\end{equation}
We have $[u'_\Xi] \cup [u'_\Psi]
= 0$, since $dz\wedge dz = 0$, so that, from \eqref{eq:realu}, 
$$
[\hat c(\la_\Xi)] \cup [\hat c(\la_\Psi)] 
= ([u'_{\Xi}]+\overline{[u'_{\Xi}]})\cup ([u'_{\Psi}]+\overline{[u'_{\Psi}]})
= 2\Re( [u'_{\Xi}]\cup \overline{[u'_{\Psi}]}).
$$
From \eqref{eq:ABG} and \eqref{eq:ShiWP}, we obtain
$$
\ABG([\hat h(\Xi)], [\hat h(\Psi)]) 
= B_*\left([\hat c(\la_\Xi)] \cup [\hat c(\la_\Psi)] \right)\cdot [M]
= 2\Re B_*( [u'_{\Xi}]\cup \overline{[u'_{\Psi}]})
= 2\WP([\hat h(\Xi)], [\hat h(\Psi)]).
$$
Because of the anti-linear isomorphism $H^0(M; K^{\otimes 2}) 
\cong H^1(M; K^{\otimes -1})$, $\Xi \mapsto [\hat h(\Xi)]$, this proves Proposition \ref{prop:WPABG}.\qed


\end{document}